\documentclass{amsart}
\usepackage[margin=1.5in]{geometry}
\usepackage{amssymb,amsmath,array,multirow,makecell,blindtext,amsthm,enumitem,mathtools,amscd,tikz-cd,amsrefs,dsfont,hyperref,url,caption,float,placeins,bm,color,mathrsfs,comment}
\usepackage[font=small, labelfont=bf]{caption}
\hypersetup{
    colorlinks=true,
    linkcolor=blue,
    citecolor=black,
    filecolor=magenta,      
    urlcolor=cyan,
}
\usepackage[all]{xy}
\usepackage{graphicx}
\usetikzlibrary{positioning}
\usepackage{caption}

 \DeclareFontFamily{U}{wncy}{}
    \DeclareFontShape{U}{wncy}{m}{n}{<->wncyr10}{}
    \DeclareSymbolFont{mcy}{U}{wncy}{m}{n}
    \DeclareMathSymbol{\Sha}{\mathord}{mcy}{"58}

    \definecolor{ForestGreen}{RGB}{34,139,34}
\newcommand{\green}[1]{\textcolor{ForestGreen}{#1}}

\theoremstyle{plain}

\newtheorem{theorem}{Theorem}[section]
\newtheorem{corollary}[theorem]{Corollary}
\newtheorem{lemma}[theorem]{Lemma}
\newtheorem{remark}[theorem]{Remark}
\newtheorem{proposition}[theorem]{Proposition}
\newtheorem{definition}[theorem]{Definition}
\newtheorem{defn}[theorem]{Definition}

\newtheorem{example}[theorem]{Example}
\numberwithin{theorem}{section}

\newcommand{\Z}{\mathbb{Z}}
\newcommand{\Q}{\mathbb{Q}}
\newcommand{\Qp}{\mathbb{Q}_p}

\newcommand{\C}{\mathbb{C}}

\newcommand{\F}{\mathbb{F}}
\newcommand{\Pp}{\mathfrak{p}}

\newcommand{\Oo}{\mathcal{O}}
\newcommand{\Gg}{\mathcal{G}}

\newcommand{\Gal}{\operatorname{Gal}}

\newcommand{\sgn}{\operatorname{sgn}}
\newcommand{\End}{\operatorname{End}}
\newcommand{\Hom}{\operatorname{Hom}}
\newcommand{\Div}{\operatorname{Div}}

\newcommand{\ord}{\operatorname{ord}}

\newcommand{\GL}{\operatorname{GL}}
\newcommand{\Fil}{\operatorname{Fil}}
\newcommand{\nf}{\normalfont}

\newcommand{\Mod}[1]{\ \mathrm{mod}\ #1}

\newcommand{\cyc}{{\mathrm{cyc}}}

\newcolumntype{?}{!{\vrule width 1pt}}

\newcommand{\fM}{\mathfrak{M}}
\newcommand{\sL}{\mathscr{L}}
\newcommand{\irr}{\mathrm{irr}}
\newcommand{\vp}{\varphi}
\newcommand{\cO}{\mathcal{O}}
\newcommand{\cG}{\mathcal{G}}
\newcommand{\RGL}[1]{{\color{blue}#1}}

\newcommand{\FF}{\mathbb{F}}
\newcommand{\QQ}{\Q}
\newcommand{\ZZ}{\Z}
\newcommand{\Zp}{\Z_p}
\newcommand{\Up}{\Upsilon}
\newcommand{\Tw}{\mathrm{Tw}}
\newcommand{\fm}{\mathfrak{m}}
\newcommand{\HH}{\mathbb{H}}
\newcommand{\Tam}{\mathrm{Tam}}

\newtheorem{lthm}{Theorem} 

\title[Mazur--Tate elements of modular forms with Serre weight $>2$]{Mazur--Tate elements of non-ordinary modular forms with Serre weight larger than two}

\author[R.~Gajek-Leonard]{Rylan Gajek-Leonard}
\address[Gajek-Leonard]{Department of Mathematics\\
Union College\\
Bailey Hall 206B\\
Schenectady, NY, 12308\\
USA}
\email{gajekler@union.edu}

\author[A.~Lei]{Antonio Lei}
\address[Lei]{
Department of Mathematics and Statistics\\
University of Ottawa\\
150 Louis-Pasteur Pvt\\
Ottawa, ON, K1N 6N5\\
Canada}
\email{antonio.lei@uottawa.ca}


\begin{document}

\begin{abstract}
Fix an odd prime $p$ and let $f$ be a non-ordinary eigen-cuspform of weight $k$ and level coprime to $p$. Assuming $p>k-1$, we compute asymptotic formulas for the Iwasawa invariants of the Mazur--Tate elements attached to $f$ in terms of the corresponding invariants of the signed $p$-adic $L$-functions. By combining this with a version of mod $p$ multiplicity one, we also obtain descriptions of the $\lambda$-invariants of Mazur--Tate elements attached to certain higher weight modular forms with Serre weight $<p+1$, generalizing results of Pollack and Weston in the Serre weight 2 case. 
\end{abstract}

\maketitle

\section{Introduction}

Fix an odd prime $p$ and let $f=\sum a_n q^n$ be a normalized cuspidal eigen-newform of weight $k\geq 2$ and level $\Gamma=\Gamma_1(N)$. We assume $p\nmid N$. Throughout this article, we fix an embedding $\overline\QQ\hookrightarrow\overline{\Qp}$ and consider $f$ as a modular form in $S_k(\Gamma,\overline{\Qp})$. 
Let $K/\Q_p$ be a finite extension that contains the image of $a_n$ (under our fixed embedding) for all $n$.  Let $\cO$ and $\FF$ denote the valuation ring and residue field of $K$, respectively. We write $\rho_f$ for Deligne's $K$-linear $G_\QQ$-representation attached to $f$ and  $\overline\rho_f$ for the associated  $\FF$-linear residual representation. We assume that $\overline\rho_f$ is absolutely irreducible.

We are concerned with the Iwasawa invariants of the $p$-adic Mazur--Tate elements $\Theta_n(f)\in \Oo[G_n]$ attached to $f$, where $G_n\cong\Z/p^n\Z$ is the Galois group of the $n$-th layer of the cyclotomic $\Z_p$-extension of $\Q$. These elements were first defined for elliptic curves in \cite{MT} and were studied for general modular forms in \cite{MTT}. They are of interest because they interpolate the algebraic part of the $L$-values of $f$ twisted by a Dirichlet character factoring through the cyclotomic $\Zp$-extension of $\QQ$. Furthermore, it is expected that these elements encode important arithmetic information on the Selmer groups attached to $f$. It has been shown under certain hypotheses that up to a constant, $\Theta_n(f)$ belongs to the Fitting ideal of the Pontryagin dual of the Selmer group of $f$ over the $n$-th layer of the cyclotomic $\Zp$-extension of $\QQ$ (see \cite{kurihara02,KK,kim22,EPW2}).

 When $f$ is ordinary at $p$ (i.e., $a_p\in\cO^\times$) and $\mu(L_p(f))=0$, the Iwasawa invariants of $\Theta_n(f)$ agree for $n\gg0$ with the counterparts of the $p$-adic $L$-function of $f$; see \cite[Proposition~3.7]{PW}.  

When $f$ is non-ordinary at $p$ (i.e., $a_p\notin\cO^\times$) and $k=2$, Pollack's plus and minus $p$-adic $L$-functions (when $a_p=0$) and Sprung's $\sharp/\flat$ $p$-adic $L$-functions (for general $a_p\notin\cO^\times$) defined in \cite{pollack03,sprung17} can be used to describe the Iwasawa invariants of $\Theta_n(f)$. Assuming that the pairs of $\mu$-invariants of the aforementioned $p$-adic $L$-functions agree, we have 
\begin{equation}
    \label{eq:wt2}
\lambda(\Theta_n(f))=q_n+\lambda^\star(f),\qquad n\gg0,
\end{equation}
where $\star\in\{+,-\}$ or $\{\sharp,\flat\}$ (depending on whether $a_p=0$ or $a_p\ne0$, and the choice of $\star$ depends on the parity of $n$), $\lambda^\star(f)$ is the $\lambda$-invariant of the $\star$ $p$-adic $L$-function attached to $f$, and 
$$
q_n=\begin{cases}p^{n-1}-p^{n-2}+\cdots +p-1 & \quad\text{$n\ge2$ even,}\\
p^{n-1}-p^{n-2}+\cdots +p^2-p & \quad\text{$n\ge3$ odd.}
\end{cases}
$$
See \cite[Theorem~4.1]{PW}.

A central goal of this article is to generalize the description in \eqref{eq:wt2} to higher-weight modular forms satisfying the Fontaine--Laffaille condition (i.e., $p>k-1$). We recall that the construction of $\sharp/\flat$ $p$-adic $L$-functions has been generalized to modular forms of arbitrary weights in \cite{LLZ0} using $p$-adic Hodge theory. In particular, a choice of basis of the Wach module attached to $\rho_f$ is required. In \cite{BFSuper}, a specific choice of basis given in the work of Berger \cite{berger04} is used to study the decomposition of the unbounded Euler systems of Beilinson--Flach elements introduced in \cite{LLZ1,KLZ1,KLZ2} into bounded $\sharp/\flat$ Euler systems. Furthermore, the Iwasawa invariants of the resulting $p$-adic $L$-functions for this basis are closely related to the Bloch--Kato--Shafarevich--Tate groups of $f$, as studied in \cite{LLZ3}, suggesting that these $p$-adic $L$-functions carry arithmetic significance despite the dependence of the choice of a Wach module basis. Our first result shows that, similar to \eqref{eq:wt2}, the Iwasawa invariants of the $\sharp/\flat$ $p$-adic $L$-functions studied in \cite{BFSuper,LLZ3} can be used to describe those of $\Theta_n(f)$. 

\begin{lthm}[Theorem~\ref{thm:FL}]\label{thmA} 
Suppose that $f\in S_k(\Gamma,\overline{\Q_p})$ is a $p$-non-ordinary cuspidal eigen-newform. If $p>k-1$ and $\mu^\sharp(f)=\mu^\flat(f)\neq \infty$ then for $n\gg0$ we have 
\begin{align*}
\mu \big(\Theta_n(f)\big)&=\mu^{\star}(f),\quad\text{and}\\
\lambda\big(\Theta_n(f)\big)&=(k-1)q_n+\lambda^\star(f),
\end{align*}
where $\star=\flat$ if $n$ is odd and $\star=\sharp$ if $n$ is even. 
\end{lthm}

As mentioned above, the Mazur--Tate elements associated with $f$ interpolate the $L$-values of $f$ twisted by Dirichlet characters. Theorem~\ref{thmA} can be used to derive an explicit formula for the $p$-adic valuations of the algebraic part of these $L$-values, as given in Theorem~\ref{thm:BK}. In addition, the Bloch--Kato conjecture predicts a link between these $L$-values and the size of both  the Bloch--Kato--Shafarevich--Tate groups and the Tamagawa numbers attached to $f$. In \cite{LLZ3}, formulae regarding the aforementioned algebraic quantities are proven. The $p$-adic valuations derived from Theorem~\ref{thmA} can be regarded as the analytic counterparts of these formulae. See \S\ref{S:BK} for a detailed discussion.

Since Mazur--Tate elements are well-suited for numerical computation, Theorem \ref{thmA} gives a method to asymptotically compute the signed Iwasawa invariants attached to non-ordinary modular forms at primes satisfying the Fontaine--Laffaille condition. We provide tables of the signed $\lambda$-invariants  for various newforms that are non-ordinary at $p\in\{5,7\}$ in \S\ref{section:data}. Our numerical calculations suggest that $\mu(\Theta_n(f))=0$ for $n\gg0$ under the Fontaine--Laffaille condition.

\begin{example}\label{ex:thmA}\nf Let $p=5$ and consider the $p$-non-ordinary newform
 $$
f= q + 3 q^{2} +  q^{4} + 15 q^{5} + O(q^6)\in S_4(\Gamma_0(27)),
 $$
with LMFDB label \href{https://www.lmfdb.org/ModularForm/GL2/Q/holomorphic/27/4/a/b/}{\texttt{27.4.a.b}}. 
 The $\mu$-invariants of $\Theta_n(f)$ for $0\leq n\leq 5$ are all zero and the $\lambda$-invariants are given in the following table: 
\begin{center}
\begin{tabular}{|c||c|c|c|c|c|c|c|}
\hline
 $n$ & 0 & 1 &2&3&4&5\\  \hline
 $\lambda(\Theta_n(f))$ & 0 & 2 &12&62&312 &1562 \\  
  \hline 
\end{tabular}
\end{center}
Thus, Theorem \ref{thmA} suggests that $\lambda^{\sharp}(f)=0$ and $\lambda^{\flat}(f)=2$. 
\end{example}

\begin{remark} \nf  In the $a_p=0$ case, a description of the Iwasawa invariants of Mazur--Tate elements in terms of Pollack's plus and minus $p$-adic $L$-functions can be found in \cite{GL}, where it is shown that for $n\gg 0$ we have
\begin{align}
\label{eq:mu_ap0}\mu \big(\Theta_n(f)\big)&=\mu^{\star}(f)+\iota^\star(f)\quad\text{and}\\
\label{eq:lambda_ap0}\lambda\big(\Theta_n(f)\big)&=(k-1)q_n+\lambda^\star(f)-\iota^\star(f)\phi(p^n).
\end{align}
Here $\star\in \{+,-\}$ depends on the parity of $n$, $\iota^\pm(f)$ is a non-negative integer which is zero at weights $2\leq k\leq p+1$, and $\phi$ is Euler's totient function. See \cite[Theorem~B]{GL}.  Theorem~\ref{thmA} generalizes the above description to non-ordinary forms of any positive slope satisfying the Fontaine--Laffaille condition. (By `slope', we mean the $p$-adic valuation of the $T_p$-eigenvalue.)  We note that the Fontaine--Laffaille hypothesis plays a similar role in the current article as the $a_p=0$ hypothesis in \cite{GL}. Both assumptions allow for a concrete decomposition of the unbounded $p$-adic $L$-functions in terms of certain products of cyclotomic polynomials and $\sharp/\flat$ (or plus/minus) $p$-adic $L$-functions: see \cite[Theorem 5.1]{pollack03} when $a_p=0$ and Theorem~\ref{thm:decomp} in the Fontaine--Laffaille setting. It is these decompositions which ultimately allow one to relate the Iwasawa invariants of the Mazur--Tate elements to those attached to the signed $p$-adic $L$-functions. The assumption on $\mu$-invariants in Theorem~\ref{thmA} comes from the fact that the latter decomposition is only explicit enough for our purposes when considered modulo the maximal ideal of $\cO$ (see Proposition \ref{prop:Cnf}). 
\end{remark}

Still assuming that $f$ is non-ordinary at $p$, suppose that the weight of $f$ is $k>2$ and its Serre weight is 2 (see Definition~\ref{def:SW}), so that $\overline\rho_f$ is isomorphic to $\overline\rho_g$ for some weight 2 form $g$. In this case, Pollack and Weston \cite[Theorems 1 and 2]{PW} proved that 
\[
\lambda(\Theta_n(f))=p^n-p^{n-1}+q_{n-1}+\lambda^\star(g),
\]
assuming that $\overline\rho_f|_{G_{\Qp}}$ is not decomposable, $\mu^\star(g)=0$, and one of the following two conditions are valid:

\vspace{0.2cm}
\begin{itemize}
    \item $k<p^2+1$ (\textit{medium weight} condition);
\vspace{0.2cm}
    \item $\ord_p(a_p)<p-1$ (\textit{small slope} condition).
\end{itemize}
\vspace{0.2cm}
  The proofs of these results rely on a mod $p$ multiplicity one result of Ribet \cite{ribet91} to compare the modular symbols of $f$ and $g$. In light of Theorem~\ref{thmA}, we extend the results of Pollack and Weston to modular forms whose Serre weight is strictly between $2$ and $p+1$ (but the weight of $f$ itself may be greater than $p+1$). Using the mod $p$ multiplicity one result of Edixhoven \cite{edixhoven92}, we can relate the modular symbols of $f$ to those of a modular form $g$ whose weight is between $2$ and $p+1$ under analogues of the medium weight or small slope condition.

\begin{lthm}[Theorem~\ref{thm:mediumweight}]\label{thmB} Let $f\in S_{k_f}(\Gamma,\overline{\Q_p})$ and $g\in S_{k_g}(\Gamma,\overline{\Q_p})$ be $p$-non-ordinary cuspidal eigen-newforms. Suppose that 
\begin{enumerate}
\item[(i)]  $2<k_g<p+1$ and $\mu^\sharp(g)=\mu^\flat(g)=0$,
\item[(ii)] $\overline\rho_f\cong \overline\rho_g$, 
\item[(iii)] $k_f<(k_g-1-\delta)(p+1)$, where $\delta\in\{0,1,2\}$ is an explicit constant that depends only on $p$ and $k_f$ (see Definition \ref{def_delta}).
 \end{enumerate}
 Then for $n\gg0$, we have
  $$
 \lambda\big(\Theta_n(f)\big)=(k_g-1)q_n+\lambda^\star(g),
 $$
 where $\star=\flat$ if $n$ is odd and $\star=\sharp$ if $n$ is even.
\end{lthm}

\begin{lthm}[Theorem~\ref{thm:small-slope}]\label{thmC}  Let $f\in S_{k_f}(\Gamma,\overline{\Q_p})$ and $g\in S_{k_g}(\Gamma,\overline{\Q_p})$ be $p$-non-ordinary cuspidal eigen-newforms. Suppose that 
\begin{enumerate}
\item[(i)]  $2<k_g<p+1$ and $\mu^\sharp(g)=\mu^\flat(g)=0$,
\item[(ii)] $\overline\rho_f\cong \overline\rho_g$,
\item[(iii)]$\ord_p(a_p(f))<k_g-2.$
 \end{enumerate}
 Then for $n\gg0$, we have 
 $$
 \lambda\big(\Theta_n(f)\big)=(k_g-1)q_n+\lambda^\star(g),
 $$
 where $\star=\flat$ if $n$ is odd and $\star=\sharp$ if $n$ is even.
 \end{lthm}

\begin{example}\label{ex:thmBC}\nf 
Let $p=7$ and consider the $p$-non-ordinary newforms $g\in S_5(\Gamma_1(4))$, $f_1\in S_{17}(\Gamma_1(4))$, and $f_2\in S_{23}(\Gamma_1(4))$, with LMFDB labels \href{https://www.lmfdb.org/ModularForm/GL2/Q/holomorphic/4/5/b/a/}{\texttt{4.5.b.a}}, \href{https://www.lmfdb.org/ModularForm/GL2/Q/holomorphic/4/17/b/b/}{\texttt{4.17.b.b}}, and \href{https://www.lmfdb.org/ModularForm/GL2/Q/holomorphic/4/29/b/b/}{\texttt{4.29.b.b}}, respectively. 
Both $f_1$ and $f_2$ are congruent to $g$ at primes over 7, and the table below gives the first few $\lambda$-invariants of the Mazur--Tate elements attached to these modular forms. 

\begin{center}
\begin{tabular}{|c||c|c||c|c|c|c|c|}
\hline
LMFDB & weight  & slope & 0 & 1 & 2 & 3 & 4 \\ 
\hline
\hline
\href{https://www.lmfdb.org/ModularForm/GL2/Q/holomorphic/4/5/b/a/}{\texttt{4.5.b.a}}&  5  & $\infty$& 0 &  3&  24&  171& 1200   \\
\hline
\href{https://www.lmfdb.org/ModularForm/GL2/Q/holomorphic/4/17/b/b/}{\texttt{4.17.b.b}}& 17    & 1/2& 0 &  3&  24&  171& 1200 \\
\hline
\href{https://www.lmfdb.org/ModularForm/GL2/Q/holomorphic/4/29/b/b/}{\texttt{4.29.b.b}} & 29    &1/2& 0 &  3&  24&  171& 1200  \\  
\hline
\end{tabular}
\end{center}
The form $g$ satisfies the Fontaine--Laffaille condition and the $\lambda$-invariants of its Mazur--Tate elements follow the pattern given by Theorem \ref{thmA}. The fact that the $\lambda$-invariants of $f_1$ and $f_2$ agree with those of $g$ can then be explained by Theorems \ref{thmB} and \ref{thmC}, respectively. We note that if condition (iii) in either theorem does not hold, then the $\lambda$-invariants may or may not exhibit the behavior described in the theorems -- see Tables \ref{table3} and \ref{table4} for examples. 
\end{example}

In \cite{CL}, the signed $p$-adic $L$-functions of two congruent modular forms of the same weight under the Fontaine--Laffaille condition have been studied. Results similar to those of Greenberg--Vatsal \cite{greenbergvatsal,EPW} for $p$-ordinary modular forms have been extended to the non-ordinary case. In particular, in \cite{CL} it is shown that if $f$ and $g$ are congruent newforms of the same level with $2<k_f=k_g<p+1$ then $\mu^\star(f)=0$ if and only if $\mu^\star(g)=0$, in which case we have $\lambda^\star(f)=\lambda^\star(g)$. By applying Theorem \ref{thmA} to such $f$ and $g$ and using mod $p$ multiplicity one, we present a different proof of this statement in Corollary \ref{cor:lamflamg}.

We remark that the equality of signed $\lambda$-invariants attached to congruent pairs of modular forms need not hold when the weights of $f$ and $g$ differ. For example, when $f$ has weight $p+1$ with $a_p(f)=0$ and $g$ has weight 2, results of \cite{GL} show that 
one pair of signed invariants agree while the other differ by $p-1$.

\section*{Acknowledgment}

The authors thank Raiza Corpuz, Daniel Delbourgo, Jeffrey Hatley, Chan-Ho Kim, Zichao Lin, Robert Pollack, Naman Pratap, and Tom Weston for interesting discussions on topics related to the article. AL's research is supported by the NSERC Discovery Grants Program RGPIN-2026-04351.

\section{Preliminaries}

\subsection{Modular symbols and Mazur--Tate elements}  Let $R$ be a commutative ring and $k\ge2$ an integer. Let $V_{k-2}(R)$ denote the space of degree $k-2$ homogeneous polynomials over $R$ in the variables $X$ and $Y$.  We define a right action of $\GL_2(R)$ on $V_{k-2}(R)$ by 
$$
P(X,Y)|\gamma=P(dX-cY,-bX+aY)
$$
for $\gamma= \begin{psmallmatrix} a &b\\ c&d\end{psmallmatrix}\in \GL_2(R)$. 

Let $\Gamma$ be a congruence subgroup. Then $\Gamma$ acts by M\"obius transformations on the group $\Delta^0=\Div^0(\mathbb{P}^1(\Q))$ of degree 0 divisors of the rational projective line, and we  define a right action of $\Gamma$ on $\xi\in \Hom(\Delta^0, V_{k-2}(R))$ by setting 
$$
(\xi|\gamma)(D)=\xi(\gamma D)|\gamma 
$$
for $\gamma\in \Gamma$ and $D\in \Delta^0$. We now define the space of $R$-valued \textbf{\textit{modular symbols}} of degree $k-2$ (with respect to $\Gamma$) by 
$$
\Hom_\Gamma(\Delta^0, V_{k-2}(R))=\{\xi:\Delta^0\rightarrow V_{k-2}(R)\mid \xi|\gamma=\xi \,\,\text{for all $\gamma\in \Gamma$}\}. 
$$
By \cite[Proposition 4.2]{ashstevens}, there is a canonical Hecke-equivariant isomorphism 
$$
\Hom_\Gamma(\Delta^0, V_{k-2}(R))\cong H^1_c(\Gamma,V_{k-2}(R)),
$$
and we henceforth identify the two spaces.

 Let $\mathcal{G}_n = \Gal(\QQ(\mu_{p^n})/\QQ)$. There is an isomorphism
$\mathcal{G}_n \to \left(\ZZ/p^n\ZZ\right)^\times$,  given by $\sigma_a \mapsto a$, where $\sigma_a$ is the automorphism $\zeta \mapsto \zeta^a$, for any $\zeta \in \mu_{p^n}$.

\begin{defn}\label{def:MT}
   Let $k\ge2$ be an integer. For a modular symbol $\varphi \in H_c^1(\Gamma_1(N), V_{k-2}(R))$, we define the associated \textbf{Mazur--Tate elements} $\vartheta_n(\varphi)$ of level $n \geq 1$ by 
    $$\vartheta_n(\varphi) = \sum_{a \in (\ZZ/p^n\ZZ)^\times} \varphi\,\bigg|\, \begin{pmatrix}
    1 & -a\\ 0 & p^n
\end{pmatrix} (\{\infty\} - \{0\})\cdot \sigma_a \in R[X,Y][\mathcal{G}_n].$$
    We write
    \[
    \vartheta_n(\varphi)=\sum_{j=0}^{k-2}\binom{k-2}{j}X^jY^{k-2-j}\vartheta_{n,j}(\varphi),
    \]
    where $\binom{k-2}{j}\vartheta_{n,j}(\varphi)\in R[\cG_n]$.
\end{defn}

If $R$ is a discrete valuation ring with field of fractions $K$, we define the norm of a modular symbol $\xi \in H^1_c(\Gamma,V_{k-2}(K))$ by 
$$
||\xi|| = \max_{D\in \Delta^0}||\xi(D)||,
$$
where $||\xi(D)||$ denotes the maximum of the absolute values of the coefficients of $\xi(D)$. Following \cite{PW}, we say that a modular symbol $\xi$ is \emph{cohomological} if 
$$
||\xi||=1. 
$$
Note that there is a unique (up to units in $R$) cohomological symbol attached to any element of $ H^1_c(\Gamma,V_{k-2}(R))$, obtained by scaling by an appropriate power of a uniformizer. 

\subsection{Mazur--Tate elements of modular forms}\label{section:MT}

Throughout this section, $f=\sum a_nq^n$ is a fixed Hecke eigen-cuspform of weight $k$, level $\Gamma_1(N)$, and nebentype $\epsilon$ where $p\nmid N$. The complex modular symbol $\xi_f\in H^1_c(\Gamma_1(N),V_{k-2}(\C))$ attached to $f$ is defined by
\begin{align*}
\xi_f(\{r\}-\{s\})=2\pi i \int_s^r f(z)(zX+Y)^{k-2}dz.
\end{align*}
 Let $K/\Q_p$ be a finite extension that contains the image of $a_n$ (under our fixed embedding) for all $n$.  The ring of integers of $K$ is denoted by $\cO$. Writing $\xi_f=\xi_f^++\xi_f^-$, where  $\xi_f^\pm$ lie in the $(\pm1)$-eigenspace for $\begin{psmallmatrix} -1 &0\\ 0&1\end{psmallmatrix}$, it follows from \cite[Proposition 5.11]{PasolPopa} 
 that there exist periods $\Omega^\pm_f\in \C$ such that the modular symbols $\xi_f^\pm/\Omega_f^\pm$ take values in $V_{k-2}(K)$, with respect to our fixed embedding $\overline \Q\hookrightarrow \overline{\Q_p}$. Let $\varphi_f^\pm\in H_c^1(\Gamma_1(N),V_{k-2}(\cO))$ denote the cohomological symbols associated to $\xi_f^\pm/\Omega_f^\pm$. Thus, $\varphi_f^+$ and $\varphi_f^-$ are well-defined up to units in $\Oo$ and depend on the choice of embedding $\overline \Q\hookrightarrow \overline{\Q_p}$.

We can decompose $\mathcal{G}_{n+1}$ as
$$
\mathcal{G}_{n+1} \cong \Delta \times G_n,
$$
where $\Delta \cong (\ZZ/p\ZZ)^\times$ and $G_n$ is a cyclic group of order $p^n$. Letting $\omega$ be the Teichm\"{u}ller character on $\Delta$, we obtain an induced map $\omega^i: \Oo[\mathcal{G}_{n+1}] \to \Oo[G_n]$ for each $0 \leq i \leq p-2$.  Let $\alpha$ and $\beta$ be the two roots of the Hecke polynomial $X^2-a_pX+\epsilon(p)p^{k-1}$ in $\overline{K}$. Let $K'=K(\alpha,\beta)$ and write $\cO'$ for the ring of integers of $K'$.

\begin{defn}\label{def:MT2}
For integers $0\le i\le p-2$, $0\le j\le k-2$, and $n\ge0$, we define
$\Theta_{n,j}(f,\omega^i)\in K[G_n]$ as the image of $\vartheta_{n+1,j}(\varphi_f^{\sgn(-1)^i})\in K[\cG_{n+1}]$ under $\omega^{i-j}$. When $n\ge 1$, we define for $\Upsilon\in\{\alpha,\beta\}$ 
\[
\Theta_{n,j}(f,\Up,\omega^i)=
\frac{1}{\Up^{n+1}}\cdot\Theta_{n,j}(f,\omega^i)-\frac{\epsilon(p)p^{k-2}}{\Up^{n+2}}\cdot\nu^n_{n-1}\Theta_{n-1,j}(f,\omega^i)\in K'[G_{n}], 
\]
where $\nu^n_{n-1}:\cO[G_{n-1}]\rightarrow\cO[G_{n}]$ is the norm map that sends $\sigma\in G_{n-1}$ to the sum of the pre-images of $\sigma$ in $G_n$ under the projection map $\pi^n_{n-1} :G_n\rightarrow G_{n-1}$. 
\end{defn}

\begin{remark}\nf  When $j=0$  we simply write $\Theta_n(f,\omega^i)$ to denote $\Theta_{n,0}(f,\omega^i)$. When $i=j=0$, we write $\Theta_n(f)$ for $\Theta_{n,0}(f,\omega^0)$. 
\end{remark}

Let $\gamma$ be a topological generator of $G_\infty=\displaystyle\lim_{\leftarrow} G_n$. There is a natural identification
\begin{equation}\label{eq:identification}
K[G_n] = K[[ \gamma-1 ]] / (\gamma^{p^n}-1) = K[[X]]/(\omega_n),
\end{equation}
where $\omega_n=(1+X)^{p^n}-1$.
Note that $G_n$ is a cyclic group of order $p^n$ generated by the image of $\gamma$ in $G_n$.
Let $a\in(\ZZ/p^n\ZZ)^\times$ and write $\overline\sigma_a$ for the image of $\sigma_a$ under the natural projection $\Gg_{n+1}\rightarrow G_n$. Under the above identification, the element $\overline\sigma_a\in K[G_n]$ can be regarded as the polynomial $(1+X)^{m(a)}$, where $m(a)$ is the unique integer such that $0\le m(a)\le p^n-1$ and $\overline\sigma_a=\gamma^{m(a)}\mod \gamma^{p^n}$.

We define $Q_{n,j}(f,\omega^i)\in K[X]$  as the image of the theta element $\Theta_{n,j}(f,\omega^{i})$ under this identification. We define $Q_{n,j}(f,\Upsilon,\omega^i)\in K'[X]$ using $\Theta_{n,j}(f,\Up,\omega^i)$ similarly. Let $u$ denote the image of $\gamma$ under the $p$-adic cyclotomic character on $\cG_\infty$.

\begin{theorem}
    For $\Upsilon\in\{\alpha,\beta\}$, there exists a unique tempered distribution $\mu_{f,\Upsilon}$ on $\cG_\infty$ such that its Amice transform is a power series $L_p(f,\Upsilon)\in K'[\Delta][[X]]$ whose $\omega^i$-isotypic component $L_p(f,\Upsilon,\omega^i)$ satisfies
    \[
    L_p(f,\Upsilon,\omega^i)(u^j(1+X)-1)\equiv Q_{n,j}(f,\Upsilon,\omega^i)\mod \omega_n
    \]
    for $0\le j\le k-2$ and $n\ge0$.
\end{theorem}

\begin{proof}
    This follows from the main results of \cite{visik76,amicevelu75,MTT}.
\end{proof}

\subsection{Iwasawa invariants} We recall the definition of Iwasawa invariants of an element in $K[G_n]$.
\begin{definition}  
Given an integer $n\ge0$ and a nonzero element $F$ of $K[G_n]$, we can write $F=\displaystyle\sum_{i=0}^{p^n-1} c_i X^i$. We define the Iwasawa invariants of $F$ by 
\begin{align*}
\mu(F)&=\min(\ord_\varpi(c_i)), \\
\lambda(F)&=\min(i:\ord_\varpi(c_i)=\mu(F)),
\end{align*}
where $\varpi$ is a uniformizer of $K$.
    By convention, we define $\mu(0)=\lambda(0)=\infty$.
    If $F\in\cO[[X]]\otimes K$, we define $\mu(F)$ and $\lambda(F)$ similarly.
\end{definition}
\begin{lemma}\label{lem:evaluation}
    Let $\zeta_{p^n}$ be a primitive $p^n$-th root of unity. If $F\in K[G_n]$ is a nonzero element such that $\lambda(F)<\phi(p^n)\ord_p(\varpi)$, then
    \[
    \ord_p\left(F(\zeta_{p^n}-1)\right)=\mu(F)\ord_p(\varpi)+\frac{\lambda(F)}{\phi(p^n)}.
    \]
\end{lemma}
\begin{proof}
    Let $F\in K[G_n]$ be a nonzero element with $\mu=\mu(F)$, $\lambda=\lambda(F)<\phi(p^{n})\ord_p(\varpi)$. We can write
    \[
    F=\varpi^\mu\left(\sum_{i>\lambda}d_iX^i+d_\lambda X^\lambda+\varpi\left(\sum_{i<\lambda} d_iX^i\right)\right),\quad\text{where } d_i\in\cO,\ d_\lambda\in\cO^\times.
    \]
    We have
    \[
    \ord_p\left(d_\lambda(\zeta_{p^n}-1)^\lambda\right)=\frac{\lambda}{\phi(p^n)}<\ord_p\left(\sum_{i>\lambda}d_i(\zeta_{p^n}-1)^i\right),
    \]
    and
    \[
    \ord_p\left(d_\lambda(\zeta_{p^n}-1)^\lambda\right)=\frac{\lambda}{\phi(p^n)}<\ord_p(\varpi)\le\ord_p\left(\varpi\left(\sum_{i<\lambda} d_i(\zeta_{p^n}-1)^i\right)\right).
    \]
    Therefore, the lemma follows from the ultrametric triangle inequality.
\end{proof}

\begin{lemma}\label{lem:iwinvlayern} Let $n\geq 0$. Let $F\in \Oo[[X]]\otimes_\Oo K$ and suppose $\lambda(F)<p^n$. Then the Iwasawa invariants of $F$ and $F\Mod \omega_n$ agree. 
\end{lemma}
\begin{proof} This follows from the same argument as  \cite[Lemma 2.2]{GCMB}. 
\end{proof}

\begin{lemma}\label{lem:modpomega}Let $n\geq 0$.  Let $F\in\Oo[[X]]\otimes K$, $G\in \Oo[X]$, and suppose $F\equiv G \Mod (\varpi,\omega_n)$. If $\deg(G)<p^n$, $\lambda(F)<p^n$, and $\mu(F)=0$, then the Iwasawa invariants of $F$ and $G$ agree. 
\end{lemma}
\begin{proof} We have $G\equiv \varpi H+F\Mod \omega_n$ for some $H\in \Oo[[X]]$.  As $\mu(F)=0$ we have $\lambda(\varpi H+F)=\lambda(F)<p^n$, and the statement now follows from the previous lemma. 
\end{proof}

\subsection{Important maps on modular symbols}
Let $\FF$ denote the residue field of $K$ and let $\overline V_{k-2}=V_{k-2}(\F)$. Define the semigroup 
$$
S_0(p) = \big\{\begin{psmallmatrix} a &b\\ c&d\end{psmallmatrix}\in M_{2}(\Z)\mid \det\begin{psmallmatrix} a &b\\ c&d\end{psmallmatrix} \neq 0, \,\, p\mid c,  \,\,p\nmid a\big\}. 
$$
Let $\FF(a^j)$ denote the $S_0(p)$-module $\FF$ where $\gamma= \begin{psmallmatrix} a &b\\ c&d\end{psmallmatrix}\in S_0(p)$ acts as multiplication by $a^j$. 

\begin{lemma} Let $k>2$ be an integer. The map \begin{align*}
    \tilde \Phi_k:\overline{V}_{k-2}&\rightarrow\FF(a^{k-2}),\\
    P(X,Y)&\mapsto P(0,1)
\end{align*}
is $S_0(p)$-equivariant and therefore induces a Hecke-equivariant map
\[
\Phi_k: H^1_c(\Gamma_1(N),\overline{V}_{k-2})\rightarrow H^1_c(\Gamma_1(Np),\FF(a^{k-2})).
\]
\end{lemma}
\begin{proof}
 If $\gamma= \begin{psmallmatrix} a &b\\ c&d\end{psmallmatrix}\in S_0(p)$ then 
$$
\tilde \Phi_k(P|\gamma)=
P(0,a)=a^{k-2}P(0,1)=\gamma\cdot\tilde \Phi_k(P). 
$$
Furthermore, since 
$$
\tilde \Phi_k\left(P|\begin{psmallmatrix} a &b\\ N&p\end{psmallmatrix}\begin{psmallmatrix} p &0\\ 0&1\end{psmallmatrix}\right)=0,
$$
where $ap-bN=1$, we see that the induced map on cohomology is Hecke-equivariant at $p$ (acting as $T_p$ on the source and $U_p$ on the target).  
\end{proof}

For a $S_0(p)$-module $V$, we write $V(j)$ for the $S_0(p)$-module $V$ where the action of $\gamma\in S_0(p)$ is multiplied by $\det(\gamma)^j$. 

\begin{lemma}
Let $k\ge p+3$ be an integer. The map \begin{align*}
   \tilde \theta_k: \overline{V}_{k-p-3}(1)&\rightarrow \overline{V}_{k-2},\\
    P(X,Y)&\mapsto (X^pY-XY^p)P(X,Y)
\end{align*}
is $S_0(p)$-equivariant and therefore induces a Hecke-equivariant map
\[
\theta_k: H^1_c(\Gamma_1(N),\overline{V}_{k-p-3})(1)\rightarrow H^1_c(\Gamma_1(N),\overline{V}_{k-2}).
\]
\end{lemma}
\begin{proof}
 If $\gamma= \begin{psmallmatrix} a &b\\ c&d\end{psmallmatrix}\in S_0(p)$ then 
\begin{align*}
\tilde \theta_k(P(X,Y))|\gamma&=\big((dX)^p(-bX+aY)-dX(-bX+aY)^p\big)P(dX,-bX+aY)\\
&=ad(X^pY-XY^p)P(dX,-bX+aY)\\
&=\tilde \theta_k(P(X,Y)| \gamma). 
\end{align*}
\end{proof}

\subsection{Galois representations and Serre weights}
In what follows, $I_p$ denotes the inertia group of $G_{\Qp}$. There exist two fundamental characters $\psi$ and $\psi'=\psi^p$ of level $2$ on $I_p$, which take values in $\FF$ (see \cite[\S2.4]{edixhoven92}). Given an integer $t$, we write $I(t)=\psi^t\oplus \psi^{pt}$. Note that $I(t)=I(pt)$ and $I(t_1)=I(t_2)$ whenever $t_1\equiv t_2\mod p^2-1$.

\begin{definition}\label{def:SW}
We define the \textbf{Serre weight} of $f$ at $p$, denoted $k(\overline{\rho}_f)$, to be $1+pa+b$ if
\[
\overline{\rho}_f|_{I_p}\cong I(a+bp)\cong I(b+ap),
\]
where $0\le a<b\le p-1$. 
\end{definition}

Since $p$ is fixed throughout this article, we shall refer to the Serre weight of $f$ at $p$ as simply \textit{the} Serre weight of $f$.

\begin{remark}\nf 
    It follows from \cite[Theorem~2.6]{edixhoven92} that if $f$ is of level $N$ and weight $k$ where $p\nmid N$ with $2\le k\le p$ and $\iota_p(a_p)\notin\cO^\times$,  the Serre weight of $f$ is $k$. When $k=p+1$, the Serre weight is $2$.
\end{remark}

\section{Mazur--Tate elements of modular forms with small weights}

Throughout this section, we assume the following hypotheses hold.

\vspace{3mm}
\begin{enumerate}
\item[\textbf{(FL)}] $p>k-1$.
\item[\textbf{(n-ord)}] $\iota_p(a_p)\notin\cO^\times$.
\end{enumerate}
\begin{remark} \nf Note that hypothesis \textbf{(FL)}  implies $\Theta_{n,j}(f,\omega^i)\in \Oo[G_n]$ since in this case the binomial coefficients $\binom{k-2}{j}$ appearing in Definition \ref{def:MT} are $p$-adic units. 
\end{remark} 

The main objective of this section is to prove Theorem~\ref{thmA} on the Iwasawa invariants of the Mazur--Tate elements of $f$.

\subsection{The logarithmic matrix and signed $p$-adic $L$-functions}

We review the definition of the logarithmic matrix attached to $f$ as studied in \cite{BFSuper}. We write $\cO[[\pi]]$ for the ring of power series in $\pi$, which is equipped with an $\cO$-linear operator $\varphi$ that sends $\pi$ to $(1+\pi)^p-1$ and an $\cO$-linear action of $\cG_\infty=\lim_{\leftarrow}\cG_n$ given by $\sigma\cdot\pi=(1+\pi)^{\chi_\cyc(\sigma)}-1$, where $\chi_\cyc$ is the cyclotomic character. The Mellin transform that sends $a\in\cO[[\cG_\infty]]$ to $a\cdot (1+\pi)$ induces an isomorphism
\[
\fM:\cO[[\cG_\infty ]]\stackrel{\sim}{\longrightarrow}\cO[[\pi]]^{\psi=0},\] where $\psi$ is a left-inverse of $\vp$.

Recall that $\gamma$ is a topological generator of $G_\infty$ and let $u=\chi_{\cyc}(\gamma)$. Consider the  automorphism 
$$
\Tw:\Oo[[X]]\rightarrow\Oo[[X]] ,\quad F(X)\mapsto  F\big(u(1+X)-1\big).
$$
For $i\in\ZZ$, define $\Tw^i$ as the map on $\cO[[X]]$ given by
\[
\Tw^i(F)=F(u^i(1+X)-1).
\]
For an integer $h\geq1$, define the polynomial
\[
\omega_{n,h} =\prod_{j=0}^{h-1}\Tw^{-j}(\omega_n).
\]

\begin{definition}
Let $q=\vp(\pi)/\pi \in\cO[[ \pi ]]$ and $\delta=p/(q-\pi^{p-1})\in\cO[[\pi]]^\times$. We define 
\[
P_f=\begin{bmatrix}
    0&\frac{-1}{\epsilon(p)q^{k-1}}\\ \delta^{k-1}&\frac{a_p(f)}{\epsilon(p)q^{k-1}}
\end{bmatrix}.
\]
For $n\ge1$, we define $C_{n,f}$ to be the $2\times2$ matrix of polynomials of degree $<(k-1)p^n$ that coincide with the image of 
\[
\fM^{-1}\left((1+\pi)\vp^n(P_f^{-1})\cdots \vp(P_f^{-1})\right)\,\bmod\,\omega_{n,k-1}.
\]

We define the matrices
\[
A_f=\begin{bmatrix}
    0&\frac{-1}{\epsilon(p)p^{k-1}}\\ 1&\frac{a_p(f)}{\epsilon(p)p^{k-1}}
\end{bmatrix},\quad Q_f=\begin{bmatrix}
    \alpha &-\beta\\ -\alpha\beta&\alpha\beta
\end{bmatrix}.
\]
\end{definition}
Note that
\begin{equation}\label{eq:diag}
    Q_f^{-1}A_fQ_f=\begin{bmatrix}
        \alpha^{-1}&0\\ 0& \beta^{-1}
    \end{bmatrix}.
\end{equation}
\begin{proposition}\label{prop:log-matrix}
    The sequence of matrices $A_f^{n+1}C_{n,f}$ converges to a matrix $M_{\log,f}$ defined over $K[[X]]$, with $$ M_{\log,f}\equiv A_f^{n+1}C_{n,f}\mod \omega_{n,k-1}.$$
\end{proposition}
\begin{proof}
    See \cite[Lemma 2.6]{BFSuper}.
\end{proof}

\begin{remark}
    \nf Even though the results in \cite{BFSuper} are stated under the hypothesis that $p>k$, the results in Section 2 of \textit{op. cit.} only rely on the construction of an explicit basis of the Wach module attached to $f$ given in \cite[Proposition~V.2.3]{berger04} for which the hypothesis \textbf{(FL)} suffices.
\end{remark}

\begin{theorem}\label{thm:decomp}\label{lem:signed-cong}
  For all $ i\in\{0,1,\dots,p-2\}$, there exist $L_p(f,\sharp,\omega^i),L_p(f,\flat,\omega^i)\in \cO[[X]]\otimes_{\cO}K$ such that 
\[
\frac{1}{\alpha-\beta}\cdot \begin{bmatrix}
   L_p(f,\alpha,\omega^i)\\ L_p(f,\beta,\omega^i)
\end{bmatrix}=
 Q_f^{-1}M_{\log,f}\begin{bmatrix}
    L_p(f,\sharp,\omega^i)\\ L_p(f,\flat,\omega^i)
\end{bmatrix}.
\]
\end{theorem}
\begin{proof}
    It follows from \cite[Proposition 2.11]{BFSuper} and the interpolation properties of the $p$-adic $L$-functions.
\end{proof}

\begin{remark}
    \nf The power series $L_p(f,\sharp,\omega^i)$ and $L_p(f,\flat,\omega^i)$ can be regarded as the $\omega^i$-isotypic components of two $p$-adic $L$-functions $L_p(f,\sharp)$ and $L_p(f,\flat)$, which are bounded measures on $\Gal(\QQ(\mu_{p^\infty})/\QQ)$, as constructed in \cite[Theorem~3.25]{LLZ0}. Furthermore, they are nonzero power series when $k\ge 3$ or when $a_p\ne0$ (see \cite[Corollary~3.29 and Remark~3.30]{LLZ0}).
\end{remark}
\subsection{Iwasawa invariants of Mazur--Tate elements}
We establish an exact relationship between the signed $p$-adic $L$-functions and the Mazur--Tate elements, enabling us to express the Iwasawa invariants of the latter through the former. We begin by studying congruence properties of the logarithmic matrix, which will play a crucial role in our proof of the main theorem.

  Define for integers $h\geq1$ the polynomials 
\begin{align*}
\Phi_{n,h} &= \prod_{j=0}^{h-1}\Tw^{-j}(\Phi_n),\\
\Phi_{n,h}^+ &= \prod_{2\le m\leq n, \text{even}}\Phi_{m,h},\\
\Phi_{n,h}^- &= \prod_{1\le m\leq n, \text{odd}}\Phi_{m,h},
\end{align*}
where $\Phi_0=X$ and $\Phi_n= \sum_{i=0}^{p-1}(1+X)^{ip^n}$ is the $p^n$-th cyclotomic polynomial evaluated at $1+X$.

\begin{lemma}\label{lem:Mellin}
    As ideals of $\cO[[X]]$, we have
    \begin{align*}
        \left(\Phi_{n,k-1}^+\right)&=\fM^{-1}\left((1+\pi)\prod_{2\le m\le n, \text{even}}\vp^m(q)^{k-1}\right),\\
                \left(\Phi_{n,k-1}^-\right)&=\fM^{-1}\left((1+\pi)\prod_{1\le m\le n, \text{odd}}\vp^m(q)^{k-1}\right).\\
    \end{align*}
\end{lemma}
\begin{proof}
We only discuss the $+$ case since the $-$ case can be proved by the same argument.
It follows from \cite[Theorem~5.4]{LLZ0} and the discussion in \cite[P.5]{LLZ3} that $\fM\left(\Phi_{n,k-1}^+\right)$ is divisible by $\prod_{2\le m\le n, \text{even}}\vp^m(q)^{k-1}$.
Furthermore, \cite[Proposition~7.2]{LZWach} says that $\fM$ sends an element $F\in\cO[[X]]$ to $(1+\pi)\vp(G)$ where $G\in\cO[[\pi]]$ has the same Iwasawa invariants as $F$. Hence, the assertion follows.\end{proof}

From now on, we fix a uniformizer $\varpi$ of $\cO$.

\begin{proposition}\label{prop:Cnf}
    For all $n\ge1$, we have $$
C_{n,f}\equiv 
\begin{cases}
 \begin{bmatrix}
0 &*\Phi_{n,k-1}^+\\
*\Phi_{n,k-1}^- &0
\end{bmatrix}\Mod \varpi \cO[[X]]&\quad\text{if $n$ is odd,}\\
\\
 \begin{bmatrix}
*\Phi_{n,k-1}^- &0\\
0&*\Phi_{n,k-1}^+
\end{bmatrix}\Mod \varpi\cO[[X]] &\quad\text{if $n$ is even},
\end{cases}
$$
where $*$ denotes a unit in $\Oo[[ X]] $. 
\end{proposition}
\begin{proof}
    We prove the case where $n$ is odd; the same argument applies to the even case. Under \textbf{(n-ord)}, we have $a_p(f)\equiv 0\mod\varpi$. Therefore, 
    \[
    \vp^n(P_f^{-1})\cdots \vp(P_f^{-1})\equiv\begin{bmatrix}
        0 & *\displaystyle\prod_{2\le m\le n, \text{even}}\vp^m(q)^{k-1}\\
        *\displaystyle\prod_{1\le m\le n,\text{odd}}\vp^m(q)^{k-1}&0
    \end{bmatrix}\mod \varpi\cO[[\pi]],
    \]
    where $*$ represents an element of $\varphi(\cO[[\pi]]^\times)$.
    Therefore, the proposition follows from Lemma~\ref{lem:Mellin}.
\end{proof}

\begin{proposition}\label{prop:Qnj-sharpflat}
    Let $i\in\{0,1,\dots, p-2\}$ and $j\in\{0,\dots,k-2\}$. We have for $n\ge1$
    \[
    \begin{bmatrix}
        Q_{n,j}(f,\omega^i)\\
       -\epsilon(p)p^{k-2} \nu_{n-1}^nQ_{n-1,j}(f,\omega^i)
    \end{bmatrix}\equiv \Tw^jC_{n,f}\begin{bmatrix}
        \Tw^j L_p(f,\sharp,\omega^i)\\
        \Tw^j L_p(f,\flat,\omega^i)
    \end{bmatrix}\mod \omega_n.
    \]
\end{proposition}
\begin{proof}
    By definition,
    \[\begin{bmatrix}
        \Tw^j L_p(f,\alpha,\omega^i)\\
        \Tw^j L_p(f,\beta,\omega^i)
    \end{bmatrix}\equiv \begin{bmatrix}
        \alpha^{-n-1}&\alpha^{-n-2}\\
\beta^{-n-1}&\beta^{-n-2}
    \end{bmatrix}  \begin{bmatrix}
        Q_{n,j}(f,\omega^i)\\
       -\epsilon(p)p^{k-2} \nu_{n-1}^nQ_{n-1,j}(f,\omega^i)
    \end{bmatrix}\mod \omega_n.
    \]
Furthermore, it follows from Proposition~\ref{prop:log-matrix} and \eqref{eq:diag} that
\begin{align*}
     \begin{bmatrix}
        \Tw^j L_p(f,\alpha,\omega^i)\\
        \Tw^j L_p(f,\beta,\omega^i)
    \end{bmatrix}&\equiv(\alpha-\beta)
Q_f^{-1}A_f^{n+1}\Tw^j C_{n,f}\begin{bmatrix}
        \Tw^j L_p(f,\sharp,\omega^i)\\
        \Tw^j L_p(f,\flat,\omega^i)
    \end{bmatrix}\\
   &\equiv (\alpha-\beta)\begin{bmatrix}
       \alpha^{-n-1}&0\\0&\beta^{-n-1}
   \end{bmatrix}Q_f^{-1}\Tw^j C_{n,f}\begin{bmatrix}
        \Tw^j L_p(f,\sharp,\omega^i)\\
        \Tw^j L_p(f,\flat,\omega^i)
    \end{bmatrix}\\
    &\equiv  \begin{bmatrix}
        \alpha^{-n-1}&\alpha^{-n-2}\\
\beta^{-n-1}&\beta^{-n-2}
    \end{bmatrix} \Tw^j C_{n,f}\begin{bmatrix}
        \Tw^j L_p(f,\sharp,\omega^i)\\
        \Tw^j L_p(f,\flat,\omega^i)
    \end{bmatrix}\mod\omega_n.
\end{align*}  
Therefore, the assertion follows after combining these congruences.
\end{proof}

\begin{definition}
    For $i\in\{0,1,\dots,p-2\}$ and $\star\in\{\sharp,\flat\}$, we define $\mu(f,\star,\omega^i)$ and $\lambda(f,\star,\omega^i)$ as the $\mu$ and $\lambda$-invariants of $L_p(f,\star,\omega^i)$.
\end{definition}

For the rest of the section, we assume that the following hypothesis holds.

\vspace{0.3cm}

\begin{itemize}
    \item[\textbf{(mu)}] For all $i\in\{0,1,\dots,p-2\}$, $\mu(f,\sharp,\omega^i)=\mu(f,\flat,\omega^i)\ne\infty$.
\end{itemize}

\vspace{0.3cm}

Under the assumption that \textbf{(mu)} holds, we write $\mu(f,\omega^i)$ for the common value $\mu(f,\sharp,\omega^i)=\mu(f,\flat,\omega^i)$. We are now ready to prove Theorem~\ref{thmA}.

\begin{theorem}\label{thm:FL}
    Suppose that \textbf{(FL)}, \textbf{(n-ord)} and \textbf{(mu)} hold. For $i\in\{0,1,\dots,p-2\}$, $j\in\{0,\dots,k-2\}$ and $n\gg0$, we have
    \[
    \mu\left(\Theta_{n,j}(f,\omega^i)\right)=\mu(f,\omega^i).
    \]
    Furthermore, 
    $$
\lambda(\Theta_{n,j}(f,\omega^{i}))=\begin{cases}
(k-1)q_n+\lambda(f,\flat,\omega^{i}) &\quad \text{if $n$ is odd,}\\
(k-1)q_n+\lambda(f,\sharp,\omega^{i}) &\quad \text{if $n$ is even.}\\
\end{cases}
$$
\end{theorem}
\begin{proof}
By definition, we may replace $\Theta_{n,j}(f,\omega^{i})$ by $Q_{n,j}(f,\omega^{i})$. Let $a=\mu(f,\omega^i)$. Let $n$ be an integer sufficiently large so that for both choices of $\star$,
\[
L_p(f,\star,\omega^i)\equiv \varpi^a P_{n}^\star\mod\omega_n
\]
for some polynomial $P_n^\star\in\cO[X]\setminus\varpi\cO[X]$. Since $C_{n,f}$ is defined over $\cO[X]$, Proposition \ref{prop:Qnj-sharpflat} implies that
\[
Q_{n,j}(f,\omega^i)\equiv \varpi^a \tilde P_n\mod\omega_n
\]
for some $\tilde P_n\in\cO[X]$. As $Q_{n,j}(f,\omega^i)$ is a polynomial of degree $<p^n$, this yields
    \[
    \varpi^{-a}Q_{n,j}(f,\omega^{i})\in\cO[X].
    \]
In particular, we have the following congruence of elements in $\cO[[X]]$:
 \[
    \begin{bmatrix}
      \varpi^{-a}  Q_{n,j}(f,\omega^i)\\
       -\epsilon(p)p^{k-2} \nu_{n-1}^n \varpi^{-a}Q_{n-1,j}(f,\omega^i)
    \end{bmatrix}\equiv \Tw^jC_{n,f}\begin{bmatrix}
        \Tw^j  \varpi^{-a}L_p(f,\sharp,\omega^i)\\
        \Tw^j  \varpi^{-a}L_p(f,\flat,\omega^i)
    \end{bmatrix}\mod \omega_n\cO[[X]]
    \]
    for $n\gg0$.
    Hence, it follows from Proposition~\ref{prop:Cnf} that
$$
\varpi^{-a}Q_{n,j}(f,\omega^i)\equiv
 \begin{cases}
\Tw^j\big(*\Phi_{n,k-1}^+\varpi^{-a}L_p(f,\flat,\omega^i)\big)\Mod (\varpi,\omega_n)&\quad\text{if $n$ is odd,}\\
\Tw^j\big(*\Phi_{n,k-1}^-\varpi^{-a}L_p(f,\sharp,\omega^i)\big)\Mod (\varpi,\omega_n)&\quad\text{if $n$ is even,}
  \end{cases}
$$
where both sides are elements in $\cO[[X]]$ and  $(\varpi,\omega_n)$ is considered as an ideal of $\cO[[X]]$.
Since multiplying by an element of $\cO[[X]]^\times$ preserves Iwasawa invariants, by \cite[Lemmas 2.9 and 3.7]{GL} we have 
\begin{align*}
\lambda\big(\Tw^j(*\Phi_{n,k-1}^\circ \varpi^{-a}L_p(f,\bullet,\omega^i))\big)&=(k-1)q_n+\lambda(f,\bullet,\omega^i)\\
\mu\big(\Tw^j(*\Phi_{n,k-1}^\circ \varpi^{-a}L_p(f,\bullet,\omega^i))\big)&=0
\end{align*}
for $(\circ,\bullet)\in \{(+,\flat),(-,\sharp)\}$ and $n$ of parity $-\circ$.  The sequence $p^n-(k-1)q_n\rightarrow \infty$ as $n\rightarrow\infty$ under assumption \textbf{(FL)}, so we can take $n\gg0$ such that $(k-1)q_n+\lambda(f,\bullet,\omega^i)<p^n$.
From Lemma \ref{lem:modpomega}, it follows that the Iwasawa invariants of $\Tw^j\left(*\Phi_{n,k-1}^\circ \varpi^{-a}L_p(f,\bullet,\omega^i)\right)$ agree with those of $\varpi^{-a}Q_{n,j}(f,\omega^i)$, which concludes the proof.
\end{proof}

 See Tables \ref{table1} and \ref{table2} for examples that illustrate the above theorem where $p$ is taken to be $5$ and $7$, respectively. In all these examples, the $\mu$-invariants are equal to zero.

\begin{corollary}\label{cor:bound}
    Let $f$ be as in Theorem~\ref{thm:FL}. For a sufficiently large integer $n$ that is odd, we have $\lambda(\Theta_{n,j}(f,\omega^i))\ge (k-1)q_n+k-2$.
\end{corollary}
\begin{proof}
    In light of the formula given by Theorem~\ref{thm:FL}, it suffices to prove that 
    \begin{equation}\label{eqn:Lflatbound}
    \lambda(f,\flat,\omega^i)\ge k-2.
    \end{equation}
    Indeed, it follows from \cite[Lemma~5.5]{LLZ0.5} that the $p$-adic $L$-function $L_p(f,\flat)$ is zero when evaluated at $\chi_\cyc^j\omega^i$, where $\chi_\cyc$ is the cyclotomic character on $\Gal(\QQ(\mu_{p^\infty})/\QQ)$, $j\in \{0,\dots, k-2\}$ and $i\in\{1,\dots,p-2\}$. This is equivalent to evaluating $L_p(f,\flat,\omega^{j+i})$ at $X=u^j-1$. In particular, we see that $L_p(f,\flat,\omega^{i})$ is divisible by $X-u^{j}+1$, where $j\in\{0,1,\dots, k-2\}\setminus\{i\}$. Therefore, $\lambda(f,\flat,\omega^i)\ge k-2$, as required.
\end{proof}
Taking into account Corollary~\ref{cor:bound}, the majority of the examples presented in Tables~\ref{table1} and \ref{table2} attain the minimum possible value of $\lambda$-invariants.

\begin{remark}\nf If $f$ has weight $p+1$, we note that the $\lambda$-invariants of the Mazur--Tate elements follow the same pattern as in Theorem \ref{thm:FL} (i.e., they have the form $(k-1)q_n+O(1)$), despite $f$ not satisfying \textbf{(FL)}. This is because when $f$ is non-ordinary of weight $k=p+1$ we have $k(\overline\rho_f)=2$ and \cite[Corollary 5.3]{PW} (together with \cite[Corollary 8.9]{sprung17}) imply that if $\mu(\Theta_{n,0}(f,\omega^i))=0$ for $n\gg0$ then
\begin{align*}
\lambda(\Theta_{n,0}(f,\omega^i))
&= pq_n+
\begin{cases} 
p-1+\lambda(g,\flat,\omega^i)\quad& \text{if $n\gg0$ is odd,}\\
\lambda(g,\sharp,\omega^i) \quad& \text{if $n\gg0$ is even.}
\end{cases}
\end{align*}
Here $g$ is a $p$-non-ordinary eigen-newform of weight 2 with $\bar\rho_f\cong \bar\rho_g$. 
Note that if Theorem \ref{thm:FL} is true at weight $k=p+1$ then one could compare the description of $\lambda$-invariants in Theorem \ref{thm:FL} with the above equality to obtain the following relation between signed $\lambda$-invariants: 
\begin{align*}
\lambda(f,\flat,\omega^i)&=p-1+\lambda(g,\flat,\omega^i),\\
\lambda(f,\sharp,\omega^i)&=\lambda(g,\sharp,\omega^i).
\end{align*}
These equations are known to hold when $a_p(f)=0$ -- see \cite[Corollary 6.1]{GL}. In particular, equation \eqref{eqn:Lflatbound} also holds at weight $k=p+1$ when $a_p(f)=0$ and the inequality is strict precisely when $g$ has positive $\lambda^\flat$-invariant. 
\end{remark}

\subsection{Implication on the Bloch--Kato conjecture}\label{S:BK}

One consequence of Theorem~\ref{thm:FL} is the following description of the $p$-adic valuation of the $L$-values of $f$:

\begin{theorem}\label{thm:BK}
   Suppose that \textbf{(FL)}, \textbf{(n-ord)} and \textbf{(mu)} hold. Assume furthermore that
   \[\ord_p\varpi>\frac{p(k-1)}{p^2-1}.\]
   Let $i\in\{0,1,\dots,p-2\}$. If $\chi$ is a Dirichlet character of conductor $p^{n+1}$, then
    \[
   \ord_p\left(p^{j(n+1)}\tau(\chi)\cdot\frac{L(f,\overline\chi,j+1)}{(-2\pi i)^j\Omega_f^\delta}\right)=\mu(f,\star, \omega^i)\ord_p(\varpi)+\frac{(k-1)q_n+\lambda(f,\star,\omega^i)}{\phi(p^n)}
    \]
    for $n\gg0$, where $\tau(\chi)$ denotes the Gauss sum of $\chi$, $\star\in\{\sharp,\flat\}$ is chosen according to the sign of $n$ and $i$ satisfies $\chi|_{\cG_1}=\omega^{i-j}$.
\end{theorem}
\begin{proof}
    By \cite[(8.6)]{MTT}, evaluating $\vartheta_{n+1,j}(\varphi_f^{\mathrm{sgn}(-1)^i})$ at $\chi$ gives an element in $\overline\Qp$ whose $p$-adic valuation is equal to 
    \[\ord_p\left(p^{j(n+1)}\tau(\chi)\cdot\frac{L(f,\overline\chi,j+1)}{(-2\pi i)^j\Omega_f^\delta}\right).\]
    This is equivalent to evaluating the polynomial $Q_{n,j}(f,\omega^i)$ at $X=\zeta_{p^n}-1$, where $\zeta_{p^n}$ is a primitive $p^n$-th root of unity. 
    
    In light of Lemma~\ref{lem:evaluation} and Theorem~\ref{thm:FL}, it suffices to show that
    \[
    (k-1)q_n+\lambda_\star<\phi(p^n)\ord_p(\varpi).
    \]
    In the case where $n$ is even, $q_n=(p^n+1)/(p+1)$. Thus, the desired inequality is equivalent to
    \[
    (k-1)\frac{1+p^{-n}}{p+1}+\lambda_\star p^{-n}<\left(1-\frac{1}{p}\right)\ord_p(\varpi),
    \]
    As $p^{-n}\rightarrow 0$ as $n\rightarrow\infty$ and our hypothesis on $\ord_p(\varpi)$ ensures that
    \[
    (k-1)\frac{1}{p+1}<\left(1-\frac{1}{p}\right)\ord_p(\varpi).
    \]
    Therefore, the desired inequality holds when $n$ is sufficiently large. The same line of argument applies to the case where $n$ is odd.
\end{proof}

Let $V_f$ be Deligne's $G_\QQ$-representation attached to $f$ given in \cite{deligne69}. Our normalization is that the Hodge--Tate weights of $V_f|_{\Qp}$ are $0$ and $1-k$, with the Hodge--Tate weight of the $p$-adic cyclotomic character equal to $1$. Let $T_f$ be the canonical $G_\QQ$-stable $\cO$-lattice in $V_f$ given in \cite[\S8.3]{kato04}. For $j\in\{0,\dots,k-2\}$, write $T=T_f(j+1)$. 

Define $\Sha(\QQ(\mu_{p^n}),T^\vee)$ to be the Bloch--Kato--Shafarevich--Tate group of the Pontryagin dual $T^\vee$ of $T$ over $\QQ(\mu_{p^n})$. That is,
\[
\Sha(\QQ(\mu_{p^n}),T^\vee)=\frac{H^1_{\mathrm{f}}(\QQ(\mu_{p^n}),T^\vee)}{H^1_{\mathrm{f}}(\QQ(\mu_{p^n}),T^\vee)_{\mathrm{div}}},
\]
where $H^1_{\mathrm{f}}(\QQ(\mu_{p^n}),T^\vee)$ is the Bloch--Kato Selmer group from \cite{blochkato} and $H^1_{\mathrm{f}}(\QQ(\mu_{p^n}),T^\vee)_{\mathrm{div}}$  denotes its maximal divisible subgroup. 

Let $\Tam(T/\QQ(\mu_{p^n}))$ denote the ($p$-primary) Tamagawa number given as in \cite[P.847]{LLZ3}. Recall that it can be decomposed into a product 
\[
\Tam(T/\QQ(\mu_{p^n}))=\prod_{i=0}^{p-2}\Tam(T/\QQ(\mu_{p^n}))^{\omega^i},
\]
where $\Tam(T/\QQ(\mu_{p^n}))^{\omega^i}$ is the Tamagawa number for the representation $T(\omega^i)$ over $\QQ(\mu_{p^{n}})^\Delta$.
Let
\[
t_{n,\omega^i}=\ord_p\left(\big|\Sha(\QQ(\mu_{p^{n}}),T^\vee)^{\omega^i}\big|\times \Tam(T/\QQ(\mu_{p^n}))^{\omega^i}\right),
\]
where $\Sha(\QQ(\mu_{p^{n}}),T^\vee)^{\omega^i}$ denotes the $\omega^i$-isotypic component of $\Sha(\QQ(\mu_{p^{n}}),T^\vee)$. It was proven on P.849 \textit{op. cit.} that if $\ord_p(a_p(f))>\frac{k-1}{2p}$,  $3\le k\le p$, $K=\Qp$, and \textbf{(FL)} and \textbf{(n-ord)} hold, then there exist integers $a_{\star,\omega^i}$ and $b_{\star,\omega^i}$ such that for $n\gg0$,
\begin{equation}
t_{n+1,\omega^i}-t_{n,\omega^i}=q_n^\star+a_{\star,\omega^i}\phi(p^n)+b_{\star,\omega^i},    \label{eq:LLZ}
\end{equation}
where $q_n^*=(k-1)q_n$ if \textbf{(mu)} holds. The integers $a_{\star,\omega^i}$ and $b_{\star,\omega^i}$ are the $\mu$-invariant and the $\lambda$-invariant (after subtracting certain explicit constants) of some signed Selmer group.    This can be regarded as the algebraic analogue of the formula given by Theorem~\ref{thm:BK}. Indeed, the Bloch--Kato conjecture (after base-change from $\QQ$ to $\QQ(\mu_{p^n})$) predicts that for $n\gg0$,
\[
   \ord_p\left(p^{j(n+1)}\tau(\chi)\cdot\frac{L(f,\overline\chi,j+1)}{\Omega_f^\delta}\right)=\frac{t_{n+1,\omega^i}-t_{n,\omega^i}}{
\phi(p^n)}.
    \]
    In other words, Theorem~\ref{thm:BK} and \eqref{eq:LLZ} are consistent with this prediction. In fact, they are equivalent to each other if the signed Iwasawa main conjecture that relates the characteristic ideal of the aforementioned signed Selmer groups to the corresponding signed $p$-adic $L$-functions holds.

\section{Mazur--Tate elements of modular forms with medium weights}

The main objective of this section is to prove Theorem~\ref{thmB}. We first prove preliminary results on the image of the map $\theta_k$. We then review Edixhoven's mod $p$ multiplicity one result, which allows us to relate the Mazur--Tate elements of two modular forms of different weights.

\subsection{Image of the theta map} \label{image_theta}
For \S\ref{image_theta}, we fix an integer $k\geq 2$.
\begin{definition}\label{def_delta}
   Let $s(k)$ be the integer in $\{1,\dots, p-1\}$ such that $k-1\equiv s(k)\mod (p-1)$ and let $k'\in \{0, \dots, p-2\}$ denote the remainder of $k-2$ divided by $p-1$. Define
    \[\delta=\delta(k,p):=
    \begin{cases}
        0&\text{if }\lfloor\frac{k-2}{p+1}\rfloor\le \lfloor \frac{k'}2\rfloor,\\
        1&\text{if }\lfloor \frac{k'}2\rfloor+1\le \lfloor\frac{k-2}{p+1}\rfloor\le \lfloor \frac{k'}2\rfloor+\frac{p-1}2,\\
        2&\text{otherwise.}
    \end{cases}
    \]
\end{definition}

\begin{remark}\label{rk:remainder}
\nf Note that $s(k)=k'+1$ and, in particular,
\[
s(k)=k-1-\left\lfloor\frac{k-2}{p-1}\right\rfloor (p-1).
\]
\end{remark}

\begin{definition}
Let $\sim$ be the equivalence relation on $\displaystyle\frac{\Z}{(p^2-1)\Z}$ defined by $t\sim pt$ for all $t$.
Define $\sL^\irr(k)$ to be the set of $\displaystyle 
    t\in\frac{\Z}{(p^2-1)\Z}/\sim$ such that there exists an eigen-newform $f$ of weight $k$ and level $\Gamma_1(N)$ with $\overline\rho_f|_{I_p}\cong I(t)$. 
Furthermore, let $\sL^\irr_\theta(k)$ denote the subset of $\sL^\irr(k)$ consisting of all elements which occur for eigen-newforms that are in the image of $\theta_k$.
\end{definition}

\begin{proposition}\label{prop:theta-image}
If $k\le p+2$, then $\sL_\theta^\irr(k)=\emptyset$.
If $p+2<k<p^2+1$, then 
    \[\sL_\theta^\irr(k)=\left\{s(k)+j(p-1):1\le j\le\left\lfloor\frac{k-2}{p+1}\right\rfloor+\delta,j\ne\left \lfloor \frac{k'+2}{2}\right\rfloor ,\left\lfloor \frac{p+k'+3}2\right\rfloor \right\}.\]
\end{proposition}
\begin{proof}
We follow the line of argument of the proof of \cite[Lemma~5.4]{PW}.
It follows from \cite[Theorem~3.4]{ashstevens} and \cite[Lemma~4.10]{PW} that 
\begin{itemize}
    \item $\sL^\irr(k)=\{s(k)\}$ and $\sL^\irr_\theta(k)=\emptyset$ for $2\le k\le p+2$;
    \item $\sL^\irr(k)=\{s(k)\}\bigcup\sL^\irr_\theta(k)$ and $\sL^\irr_\theta(k)=\{t+p+1:t\in \sL^\irr(k-(p+1))\}$ for $k>p+2$.
\end{itemize}
Therefore, an inductive argument shows that
\begin{equation}\label{LirrTheta}
\sL_\theta^\irr(k)=\left\{s\left(k-i(p+1)\right)+i(p+1):1\le i\le\left\lfloor\frac{k-2}{p+1}\right\rfloor \right\}.
\end{equation}

Let us write $k-2=q(p-1)+k'$, where $q\in\Z$ and $0\le k'<p-1$. By Remark~\ref{rk:remainder}, we have
\begin{align*}
    s\left(k-i(p+1)\right)+i(p+1)&=k-1-\left\lfloor\frac{k-i(p+1)-2}{p-1}\right \rfloor(p-1)\\
    &=q(p-1)+k'+1-\left\lfloor q+\frac{k'-i(p+1)}{p-1}\right\rfloor (p-1)\\
    &=k'+1-\left\lfloor \frac{k'-i(p+1)}{p-1}\right\rfloor (p-1)\\
    &=k'+1+\left(i+\left\lceil \frac{2i-k'}{p-1}\right\rceil\right) (p-1).
\end{align*}

Since $i\le (k-2)/(p+1)$ and $k<p^2+1$, we have $i<p-1$. When $k'\in\{0,1\}$, we have $\displaystyle\left\lceil \frac{2i-k'}{p-1}\right\rceil=1$ for $1\le i\le (p-1)/2$. Otherwise, $\displaystyle\left\lceil \frac{2i-k'}{p-1}\right\rceil=2$. Therefore, it follows from \eqref{LirrTheta} that
\[
\sL_\theta^\irr(k)=\left\{s(k)+j(p-1):2\le j\le \left\lfloor\frac{k-2}{p+1}\right\rfloor+1\right\}
\]
if $(p-1)/2\ge\displaystyle\left\lfloor\frac{k-2}{p+1}\right\rfloor$.
Otherwise,
\[
\sL_\theta^\irr(k)=\left\{s(k)+j(p-1):2\le j\le \left\lfloor\frac{k-2}{p+1}\right\rfloor+2,j\ne \frac{p+3}{2} \right\}.
\]

Similarly, if $k'\ge2$, then 
\[
    \left\lceil \frac{2i-k'}{p-1}\right\rceil=
\begin{cases}
0&1\le i\le\lfloor k'/2\rfloor,\\
1&\lfloor k'/2\rfloor+1\le i\le \lfloor (p-1+k')/2\rfloor,\\
2&\text{otherwise}.
\end{cases}
\]
Therefore, if $\displaystyle\left\lfloor\frac{k-2}{p+1}\right\rfloor\le \lfloor k'/2\rfloor$, then 
\[
\sL_\theta^\irr(k)=\left\{s(k)+j(p-1):1\le j\le \left\lfloor\frac{k-2}{p+1}\right\rfloor\right\}.
\]
If $\displaystyle\lfloor k'/2\rfloor+1\le \left\lfloor\frac{k-2}{p+1}\right\rfloor\le \lfloor (p-1+k')/2\rfloor$, then
\[
\sL_\theta^\irr(k)=\left\{s(k)+j(p-1):1\le j\le \left\lfloor\frac{k-2}{p+1}\right\rfloor+1,j\ne \lfloor k'/2\rfloor+1 \right\}.
\]
If $\displaystyle \left\lfloor\frac{k-2}{p+1}\right\rfloor\ge \lfloor (p-1+k')/2\rfloor+1$, then $\sL_\theta^\irr(k)$ is given by
\[
\left\{s(k)+j(p-1):1\le j\le \left\lfloor\frac{k-2}{p+1}\right\rfloor+2,j\ne \lfloor k'/2\rfloor+1 ,\lfloor (p+3+k')/2\rfloor\right\}.
\]
This concludes the proof.\end{proof}

We are interested in whether $s(k)$ is equivalent to an element of $\sL_\theta^\irr(k)$. This is the same as asking whether
\[
 s(k)+j(p-1)\equiv s(k)\quad\text{or}\quad ps(k)\mod p^2-1
\]
for some $j$ as described by Proposition~\ref{prop:theta-image}.
This in turn is equivalent to
\[
j\equiv0\quad\text{or}\quad s(k)\mod p+1.
\]
For $k<p^2+1$, and $1\le j\le \displaystyle\left\lfloor\frac{k-2}{p+1}\right\rfloor+\delta$, the first congruence cannot occur.
When $k'=0$ so that $s(k)=1$, we see that $j\equiv s(k)\mod p+1$ does not hold since $j$ is at least $2$. When $k'=1$, $j=s(k)=2$ can happen as soon as $k\ge 2p$. Indeed, when $k=2p$, 
$\sL_\theta^\irr(k)=\{2p\}$, and $2p$ represents the same representation of $I_p$ as $2$. For $k'\ge 2$, we can guarantee that $j\equiv s(k)\mod p+1$ does not occur if 
\begin{equation}\label{weight_bound}
\left\lfloor\frac{k-2}{p+1}\right\rfloor+\delta<s(k).
\end{equation}

\begin{example}\nf When $p=5$ and $k=8$, we have $k'=2$ and $s(k)=3$. In this case,
\[
\sL_\theta^\irr(k)=\left\{s(k)+(p-1)\right\}=\left\{7\right\},
\]
and $7$ is not congruent to $3$ or $15$ modulo $p^2-1=24$. Thus, $s(k)$ is not equivalent to an element of $\sL_\theta^\irr(k)$.
The same can be said when $p=5$ and $k=12$. However, if $p=5$ and $k=16$,
\[
\sL_\theta^\irr(k)=\left\{s(k)+j(p-1):j=1,3\right\}=\{7,15\}.
\]
In this case, \eqref{weight_bound} does not hold, and $\sL_\theta^\irr(k)$ contains an element that is equivalent to $3$. 
\end{example}

\begin{corollary}\label{cor:nonzero-modp}
Let $f\in S_k(\Gamma_1(N),\overline{\Qp})$ such that  
\begin{itemize}
        \item[(i)]$\overline{\rho}_f$ is irreducible and $\overline\rho_f|_{G_{\Qp}}$ is not decomposable,
        \item[(ii)] $2<k(\overline\rho_f)<p+1$ and $k$ satisfies the inequality \eqref{weight_bound}.
    \end{itemize}
    Then  $\Phi_k(\overline\vp_f^\pm)\ne0$.
\end{corollary}
\begin{proof}
    This follows from combining Proposition~\ref{prop:theta-image} and \cite[Theorem~3.4(a)]{ashstevens}, the latter of which tells us that $\Phi_k(\overline\vp_f^\pm)=0$ if and only if $\overline \vp_f^\pm$ belongs to the image of $\theta_k$.
\end{proof}

\begin{example} \nf Let $p=7$ and let $k\in \{6,12,18, 24, 30, 36\}$. We compute using Proposition~\ref{prop:theta-image} that
\begin{itemize}
\item $\sL_\theta^\irr(6)=\emptyset$
\item $\sL_\theta^\irr(12)=\{11\}$
\item $\sL_\theta^\irr(18)=\{11,17\}$
\item $\sL_\theta^\irr(24)=\{11,17\}$
\item $\sL_\theta^\irr(30)=\{11,17,19\}$
\item $\sL_\theta^\irr(36)=\{11,17,19,35\}$
\end{itemize}
Note that $s(k)=5$ and $ps(k)=35$ are not congruent mod $p^2-1=48$ to any element of  $\sL_\theta^\irr(k)$, except when $k=36$. This implies that $\Phi_k(\overline\varphi_f^\pm)\neq 0$ for each modular form in the last block of Table \ref{table4}, except the one with weight 36. 

\end{example}

\subsection{Proof of Theorem~\ref{thmB}}

We now review Edixhoven's mod $p$ multiplicity one result.

\begin{theorem}\label{thm:Ed}
    Let $g\in S_k(\Gamma_1(N),\epsilon,\overline\Qp)$ be a normalized eigen-newform such that $p\nmid N$ and $2< k< p+1$. Let $J_\QQ$ be the Jacobian of the modular curve $X_1(pN)_{\QQ}$. Let $\HH$ be the subring of $\End(J_\QQ)$ generated by the Hecke operators $T_\ell$ and the Diamond operators $\langle a\rangle_N$, $a\in(\ZZ/N\ZZ)^\times$ and $\langle b\rangle_p$, $b\in(\ZZ/p\ZZ)^\times$. Let $\fm$ be the maximal ideal of $\HH$ corresponding to the eigenvalues of $g$ modulo $\varpi$ and let $\FF=\HH/\fm$. Suppose that $a_p\notin\cO^\times$. Then $J_\QQ(\overline\QQ)[\fm]$ is an $\FF$-vector space of dimension 2.
\end{theorem}
\begin{proof}
Note that our assumptions on $k$ and $a_p$ ensure that $\overline\rho_g|_{G_{\Qp}}$ is irreducible by \cite[Theorem~2.6]{edixhoven92} and $a_p^2\not\equiv\epsilon(p)\mod\varpi$.  Therefore, this is a special case of \cite[Theorem~9.2]{edixhoven92}. 
\end{proof}

\begin{corollary}\label{cor:Ed}
    Let $f\in S_{k_f}\left(\Gamma_1(N),\overline{\Qp}\right)$ and $g\in S_{k_g}\left(\Gamma_1(N),\overline\Qp\right)$. Suppose that $\overline\rho_f\cong \overline\rho_g$ and that $g$ satisfies the hypotheses of Theorem~\ref{thm:Ed}. Then 
    there exist  constants $c^\pm\in\FF=\HH/\fm$ such that
    $$
\Phi_{k_f}(\overline\vp_f^\pm)= c^\pm\cdot\Phi_{k_g}(\overline\vp_g^\pm).
 $$
\end{corollary}
\begin{proof}
In this proof, we write $\Gamma=\Gamma_1(Np)$. Let $\fm$ denote the maximal ideal defined in the statement of Theorem~\ref{thm:Ed}.
   As explained in the discussion given in \cite[P.411]{Vat99}, after applying the Albanese map combined with the Poincar\'e duality, Theorem~\ref{thm:Ed} implies that the $\fm$-component of the parabolic cohomology group $H^1_P(\Gamma,\cO)_\fm$ is an $\FF$-vector space of dimension two. This implies in turn that the complex conjugation eigenspace $H^1_P(\Gamma,\cO)_\fm^\pm$ is an $\FF$-vector space of dimension one. 
   
   Recall that  $H^1_P(\Gamma,\cO)$ coincides with the image of the compactly supported cohomology group $H^1_c(\Gamma,\cO)$ in $H^1(\Gamma,\cO)$ (see the discussion in \cite[P.404]{Vat99}). As the kernel of this map corresponds to boundary symbols and we assume that $\overline\rho_f$ is irreducible, we deduce that $H^1_c(\Gamma,\cO)^\pm_\fm$ is of dimension one over $\FF$.

   According to \cite[Theorem~3.4(a)]{ashstevens}, the kernel of $\Phi_{k_g}$ is trivial if $2< k_g<p+1$. Therefore, $\Phi_{k_g}(\overline\vp_g^\pm)$ is an $\FF$-basis of  $H^1_c(\Gamma,\cO)^\pm_\fm$. The hypothesis that $\overline\rho_f\cong \overline\rho_g$ implies that $\Phi_{k_f}(\overline\vp_f^\pm)$ belongs to $H^1_c(\Gamma,\cO)^\pm_\fm$, from which the corollary follows.
\end{proof}

\begin{theorem}
\label{thm:medium-weight}
 Let $f\in S_{k_f}\left(\Gamma_1(N),\overline{\Qp}\right)$ and $g\in S_{k_g}\left(\Gamma_1(N),\overline\Qp\right)$ be $p$-non-ordinary cuspidal eigen-newforms such that 
\begin{enumerate}
\item[(i)]  $2<k_g<p+1$,
\item[(ii)] $\overline\rho_f\cong \overline\rho_g$, 
\item[(iii)] $k_f<(k_g-1-\delta)(p+1)$, where $\delta=\delta(k_f,p)$ is as in Definition \ref{def_delta}.
 \end{enumerate}
  Then there exists a choice of cohomological periods $\Omega_f^\pm$ and $\Omega_g^\pm$ such that
 $$
\Phi_{k_f}(\overline\vp_f^\pm)= \Phi_{k_g}(\overline\vp_g^\pm).
 $$
\end{theorem}
\begin{proof}
By \cite[Theorem~2.6]{edixhoven92}, we have $\overline\rho_f|_{G_{\Qp}}=\overline\rho_g|_{G_{\Qp}}$ is irreducible and $\overline\rho_f|_{I_p}=\overline\rho_g|_{I_p}=I(k_g-1)$. We have already seen in the proof of Corollary~\ref{cor:Ed} that $\Phi_{k_g}(\overline\vp_g^\pm)\ne0$. Corollary~\ref{cor:nonzero-modp} and conditions (i)--(iii) ensure that $\Phi_{k_f}(\overline\vp_f^\pm)\ne0$ (note that (i) and (ii) imply $s(k_f)=k_g-1$). Therefore, the constant $c$ given by Corollary~\ref{cor:Ed} is nonzero. After multiplying the periods by a $p$-adic unit if necessary, we can take $c^\pm=1$.
\end{proof}

We have the following immediate corollary:
\begin{corollary}\label{cor:ThmB}
    Let $f$ and $g$ be modular forms satisfying the hypotheses of Theorem \ref{thm:medium-weight}. Then for all $n\geq 0$ we have $
\lambda(\Theta_{n}(f,\omega^i))=\lambda(\Theta_{n}(g,\omega^i))$
and  
$$
\mu(\Theta_{n}(f,\omega^i))=0 \quad \text{if and only if} \quad \mu(\Theta_{n}(g,\omega^i))=0.
$$
\end{corollary}

Combined with Theorem~\ref{thm:FL}, we can use the above corollary to relate the signed Iwasawa invariants of congruent modular forms of the same weight and level. 

\begin{corollary}\label{cor:lamflamg} Let $f,g\in S_{k}(\Gamma,\overline{\Q_p})$ be $p$-non-ordinary cuspidal eigen-newforms. Suppose that $2<k<p+1$ and $\overline\rho_f\cong \overline\rho_g$. Then $\mu(f,\sharp,\omega^i)=\mu(f,\flat,\omega^i)=0$ if and only if $\mu(g,\sharp,\omega^i)=\mu(g,\flat,\omega^i)=0$,
and if either of these conditions hold then $\lambda(f,\star,\omega^i)=\lambda(g,\star,\omega^i)$ for $\star\in \{\sharp,\flat\}$.
\end{corollary} 
\begin{proof} Without loss of generality, suppose $\mu(f,\sharp,\omega^i)=\mu(f,\flat,\omega^i)=0$. Then Theorem~\ref{thm:FL} implies that for $n\gg0$, 
\begin{align}
\label{lam1}\lambda\big(\Theta_n(f,\omega^i)\big)&=(k-1)q_n+\lambda(f,\star,\omega^i), \quad\text{and}\\
\label{mu1} \mu\big(\Theta_n(f,\omega^i)\big)&=0.
\end{align}
The second of these equations together with Corollary \ref{cor:ThmB} implies $\mu\big(\Theta_n(g,\omega^i)\big)=0$ for $n\gg0$. From Propositions \ref{prop:Cnf} and \ref{prop:Qnj-sharpflat}, we have
$$
Q_{n,j}(g,\omega^i)\equiv
 \begin{cases}
*\Phi_{n,k-1}^+L_p(g,\flat,\omega^i)\Mod (\varpi,\omega_n)&\quad\text{if $n$ is odd,}\\
*\Phi_{n,k-1}^-L_p(g,\sharp,\omega^i)\Mod (\varpi,\omega_n)&\quad\text{if $n$ is even,}
  \end{cases}
$$
where $*$ denotes a unit in $\Oo[[X]]$. In particular, if $\mu(g,\sharp,\omega^i)$ or $\mu(g,\flat,\omega^i)$ were positive then the right side of the above congruence would be zero, a contradiction as $\mu\big(\Theta_n(g,\omega^i)\big)=0$ and $Q_{n,j}(g,\omega^i)$ has degree $<p^n$. Thus $\mu(g,\sharp,\omega^i)=\mu(g,\flat,\omega^i)=0$ and we can now apply Theorem \ref{thm:FL} to $g$ to obtain 
$$
\lambda\big(\Theta_n(g,\omega^i)\big)=(k-1)q_n+\lambda(g,\star,\omega^i).
$$
The equality between signed $\lambda$-invariants now follows from Corollary \ref{cor:ThmB} and equation \eqref{lam1}. 
\end{proof}

\begin{remark}
  \nf   A similar result was obtained in \cite{CL}, where $f$ and $g$ are assumed to have isomorphic residual representations (but not necessarily of the same level) and the coefficient field $K$ is unramified over $\Qp$. In this setting, the $\lambda$-invariants are not necessarily exactly the same but differ by constants arising from Euler factors at primes dividing the levels, similar to the results in the ordinary case given in \cite{greenbergvatsal}.
\end{remark}

Theorem~\ref{thmB}, which is a special case of the following theorem (after taking $i=0$), follows from combining Theorem~\ref{thm:FL} and Corollary~\ref{cor:ThmB}.

\begin{theorem}\label{thm:mediumweight}
    Let $f\in S_{k_f}(\Gamma,\overline{\Q_p})$ and $g\in S_{k_g}(\Gamma,\overline{\Q_p})$ be $p$-non-ordinary cuspidal eigen-newforms, and $i\in\{0,1,\dots, p-2\}$. Suppose that 
\begin{enumerate}
\item[(i)]  $2<k_g<p+1$ and $\mu(g,\sharp,\omega^i)=\mu(g,\flat,\omega^i)=0$,
\item[(ii)] $\overline\rho_f\cong \overline\rho_g$, 
\item[(iii)] $k_f<(k_g-1-\delta)(p+1)$, where $\delta=\delta(k_f,p)$ is as in Definition \ref{def_delta}.
 \end{enumerate}
 Then for $n\gg0$, we have
  $$
 \lambda\big(\Theta_n(f,\omega^i)\big)=(k_g-1)q_n+\lambda(g,\star,\omega^i),
 $$
 where $\star=\flat$ if $n$ is odd and $\star=\sharp$ if $n$ is even.
\end{theorem}
The reader is referred to Tables~\ref{table3} and \ref{table4} for examples that illustrate this theorem. In the first block of Table~\ref{table3}, all modular forms are congruent to each other, with $p=5$ and $k_g=4$. Note that the form \texttt{G0N9k16A} does not satisfy condition (iii), and the $\lambda$-invariants of the corresponding Mazur--Tate elements exhibit different behavior from the forms of smaller weight that satisfy condition (iii). However, while the form \texttt{G0N9k16E} also does not satisfy condition (iii), its Mazur--Tate elements have the same $\lambda$-invariants as those of smaller weight. Similar phenomena can be observed in the other blocks of the two tables.

\begin{remark}\nf It is curious to note that any modular form $f$ with $k_f> p+1$ satisfying the assumptions of Theorem \ref{thm:mediumweight} must have $a_p(f)\neq 0$. This is particularly apparent in Table \ref{table4}, where one sees that the first form in each block which exhibits a different pattern in its $\lambda$-invariants always has $a_p=0$. This can be explained as follows: if $f$ and $g$ are as in the theorem and $a_p(f)=0$ then the description of $\lambda$-invariants in Theorem \ref{thm:mediumweight} combines with that given by equation \eqref{eq:lambda_ap0} to yield 
$$
\lambda(f,\star)-\lambda(g,\star)=\iota^\star(f)\phi(p^n)-(k_f-k_g)q_n.
$$
Since $q_n(p^2-1)=\phi(p^{n+1})-\dagger$, where $\dagger\in \{p-1,p^2-p\}$ depending on whether $n$ is even or odd, we can simplify the right side of the above equation to get 
$$
\lambda(f,\star)-\lambda(g,\star)=\phi(p^n)\bigg(\iota^\star(f)-\frac{p(k_f-k_g)}{p^2-1}\bigg)+\frac{\dagger(k_f-k_g)}{p^2-1}. 
$$
In particular, since $\iota^\star(f)$ is an integer, for the right side to be constant (over $n$ of fixed parity) we must have $k_f\equiv k_g\Mod p^2-1$, but hypotheses (i) and (iii) imply $k_f-k_g<p^2-1$. 
\end{remark}

\section{Mazur--Tate elements of modular forms with small slope}

The goal of this section is to prove Theorem~\ref{thmC}.
Let $f$, $K$, and $\Oo$ be as in \S\ref{section:MT}. Let $\F$ denote the residue field of $K$. In this section we write $V_{k-2}=V_{k-2}(\Oo)$ and $\overline V_{k-2}=V_{k-2}(\F)$.  Following \cite{PW}, we define for $r\geq 0$  the filtration 
$$
\Fil^r(V_{k-2})=\bigg\{\sum_{j=0}^{k-2} b_jX^jY^{k-2-j}\in V_{k-2}\, \big|\,  \text{$\ord_p(b_j)\geq r-j$ for all $0\leq j\leq r-1$}\bigg\}. 
$$
We will also need the following subfiltration of $\Fil^r(V_{k-2})$, defined for $t\leq r$ by 
$$
\Fil^{r,t}(V_{k-2})=\bigg\{\sum_{j=0}^{k-2} b_jX^jY^{k-2-j}\in \Fil^{r}(V_{k-2})\, \big|\,  \text{$\ord_p(b_j)\geq r-j+1$ for all $r+1-t\leq j\leq r$}\bigg\}. 
$$
In what follows we write $(\Oo/p\Oo(a^j))(r)$ for the $S_0(p)$-module $\Oo/p\Oo$ on which $\gamma= \begin{psmallmatrix} a &b\\ c&d\end{psmallmatrix}\in S_0(p)$ acts by $\det(\gamma)^r a^j$.

\begin{lemma}\label{lem:filtration} \phantom{}
\begin{enumerate}
\item The filtrations $\Fil^r(V_{k-2})$ and $\Fil^{r,t}(V_{k-2})$ are stable under the action of $S_0(p)$. 
\item If $P\in \Fil^r(V_{k-2})$ then $P|T_p\in p^rV_{k-2}$. 
\item For all $0\leq t\leq r$, the map 
$$
\Fil^{r,t}(V_{k-2})\rightarrow \Oo/p\Oo,\qquad \sum_{j=0}^{k-2} b_jX^jY^{k-2-j}\mapsto p^{-t}b_{r-t}
$$
induces the following isomorphism of $S_0(p)$-modules:
$$
\Fil^{r,t}(V_{k-2})/\Fil^{r,t+1}(V_{k-2})\cong (\Oo/p\Oo(a^{k-2-2r+2t}))(r-t).
$$
In particular, the quotient is generated by the image of $p^tX^{r-t}Y^{k-2-r+t}$.
\item Let $\varphi\in H^1_c(\Gamma_1(N),\Fil^r(V_{k-2}))$ be a cohomological $T_p$-eigensymbol with $T_p$-eigenvalue $\lambda$. Then $\ord_p(\lambda)\geq r$. 
\end{enumerate}
\end{lemma}
\begin{proof}
Parts (1)  and (3) follow from \cite[Lemma 6.4 and 6.6]{PW}. Part (2) follows from the fact that 
$$
p^{r-j}X^jY^{k-2-j}|\begin{psmallmatrix} 1 &a\\ 0&p\end{psmallmatrix}\in p^rV_{k-2}\qquad \text{and}\qquad p^{r-j}X^jY^{k-2-j}|\begin{psmallmatrix} u &v\\ N&p\end{psmallmatrix}\begin{psmallmatrix} p &0\\ 0&1\end{psmallmatrix}\in p^rV_{k-2},
$$
where $0\leq a\leq p-1$ and $up-vN=1$.
  For part (4), choose $D\in \Delta^0$ so that $||\varphi(D)||=1$ and note that (2) implies 
\[
\lambda\varphi(D)=(\varphi|T_p)(D)=\sum_{a=0}^{p-1}\big(\varphi|\begin{psmallmatrix} 1 &a\\ 0&p\end{psmallmatrix}\big)(D)+\big(\varphi|\begin{psmallmatrix} u &v\\ N&p\end{psmallmatrix}\begin{psmallmatrix} p &0\\ 0&1\end{psmallmatrix}\big)(D)\in p^rV_{k-2}. 
\]
\end{proof}

As noted in \cite{PW}, there is a convenient lower bound on the $\mu$-invariants of Mazur--Tate elements at $j=0$ given by taking the minimum valuation of the coefficients of $Y^{k-2}$ in the associated modular symbol. Specifically, defining
$$
\mu_{\text{min}}^\pm(f) = \min_{D\in \Delta^0}\ord_p\bigg(\varphi^\pm_f(D)\bigg|_{(X,Y)=(0,1)}\bigg),
$$
we have
$$
\mu_{\text{min}}^{\sgn(-1)^i}(f)\leq \mu(\Theta_{n,0}(f,\omega^i))
$$
for all $n$ since the coefficients of the Mazur--Tate elements are defined in terms of certain values of the associated modular symbol. (This bound need not hold for the Mazur--Tate elements at $j\neq 0$ since they are constructed using different coefficients of the modular symbol.) Note that $\mu_{\text{min}}^\pm(f)\neq \infty$ by \cite[Lemma 4.3]{PW}. 

\begin{theorem} 
\label{thm:smallslope}
Let $f\in S_{k_f}(\Gamma_1(N),\overline{\Q}_p)$ and $g\in S_{k_g}(\Gamma_1(N),\overline{\Q}_p)$ be $p$-non-ordinary cuspidal eigen-newforms. Suppose that 
\begin{enumerate}
\item[(i)]  $2<k_g<p+1$,
\item[(ii)] $\overline\rho_f\cong\overline\rho_g,$ 
\item[(iii)]$\ord_p(a_p(f))<k_g-2.$
 \end{enumerate}
Then there exists a choice of cohomological periods $\Omega_f^\pm$ and $\Omega_g^\pm$ such that
 $$
\varpi^{-\mu_{\text{\nf min}}^\pm(f)}\Phi_{k_f}(\overline\vp_f^\pm)= \Phi_{k_g}(\overline\vp_g^\pm).
 $$
\end{theorem}
\begin{proof} We follow the argument of \cite[proof of Theorem 6.1]{PW}. Let $k=k_g$ and $\varphi_*=\varphi_*^\pm$ for $*\in \{f,g\}$. Let $r\geq0$ be the largest integer for which $\varphi_f$ takes values in $\Fil^r=\Fil^r(V_{k_f-2})$ and let $t\geq 0$ be the largest integer for which $\varphi_f$ takes values in $\Fil^{r,t}=\Fil^{r,t}(V_{k_f-2})$. 
By Lemma \ref{lem:filtration}.(4) we have $r\leq \ord_p(a_p(f))$ and therefore hypothesis (iii) implies
$$
t\leq r< k-2. 
$$

Consider the image of $\varphi_f$ under the composition 
$$
H^1_c(\Gamma_1(N), \Fil^{r,t})\rightarrow H^1_c(\Gamma_1(pN), \Fil^{r,t}/\Fil^{r,t+1})\rightarrow H^1_c(\Gamma_1(pN), \Oo/p\Oo(a^{k-2-2r+2t}))(r-t)
$$
where the first map is induced by restriction to $\Gamma_1(pN)$ and projection to $\Fil^{r,t}/\Fil^{r,t+1}$ and the second is induced by the isomorphism of Lemma \ref{lem:filtration}.(3) and the fact that $k_f\equiv k\Mod p-1$. Since $\varphi_f$ does not take  all its values in $\Fil^{r,t+1}$, the image of $\varphi_f$ under this composition is nonzero. 
Letting $m\geq 0$ denote the largest integer such that all values of the image of  $\varphi_f$ are in $\varpi^m\Oo$ but not $\varpi^{m+1}\Oo$,  projecting to $\varpi^{m+1}\Oo$ and multiplying by $\varpi^{-m}$ yields a nonzero symbol
$$
\overline\eta_f\in H^1_c(\Gamma_1(pN),\F(a^{k-2-2r+2t}))(r-t). 
$$
It follows from \cite[Proposition 2.5 and Lemma 2.6]{ashstevens} that there exists a weight two eigen-newform $h\in S_2(\Gamma_1(pN),\omega^{k-2-2r+2t})$ with 
\begin{equation}\label{eqn:fcongh}
\overline\rho_f\cong\overline\rho_h\otimes\omega^{r-t}.
\end{equation}
By hypothesis (ii) and \cite[Theorem~2.6]{edixhoven92}, we know that $\overline\rho_f|_{G_{\Q_p}}$ is irreducible and 
$$
\overline\rho_f|_{I_p}\cong I(k-1).
$$
Furthermore, \cite[Lemma 4.10]{PW} implies
\[
\overline\rho_h|_{I_p}\cong I\left(s(k-2r+2t)\right). 
\]
It now follows from \eqref{eqn:fcongh} that
\begin{equation}
    I\left(s(k-2r+2t)+(p+1)(r-t)\right)\cong I(k-1).\label{eq:iso-I}
 \end{equation}
Since $0\le t\le r<k-2 $, we have 
\[
-k+3< k-1-2(r-t)\le k-1,
\]
and as $0<k-1\le p-1$, this implies
\[
-(p-1)< k-1-2(r-t)\le p-1.
\]
By definition, 
\[
s(k-2r+2t)\equiv k-1-2(r-t)\mod (p-1),
\]
thus the above inequalities imply that $s(k-2r+2t)$ is either equal to $k-1-2(r-t)$ or $k-2-2(r-t)+p$. We claim that the first case forces $r=t$ and the second does not occur.

If $s(k-2r+2t)=k-1-2(r-t)$, the isomorphism \eqref{eq:iso-I} implies that
$$
k-1+(p-1)(r-t)\equiv k-1\quad\text{or}\quad p(k-1)\mod p^2-1,
$$
which forces $r-t\equiv 0$ or $k-1\mod p+1$. But $0\le r-t<k-2<p+1$, we must have  $r=t$. If $s(k-2r+2t)=k-2-2(r-t)+p$, the isomorphism \eqref{eq:iso-I} implies that
$$
k-2+(p-1)(r-t)+p\equiv k-1\quad\text{or}\quad p(k-1)\mod p^2-1,
$$
which forces $r-t+1\equiv 0$ or $k-1\mod p+1$. But $1\le r-t+1<k-1<p+1$, the congruence above cannot hold.  

Thus $r=t$ and therefore $\varphi_f$ takes values in $\Fil^{r,r}$.  By Lemma \ref{lem:filtration}.(3), the image of $p^rY^{k-2}$ generates $F^{r,r}/F^{r,r+1}$, hence the symbol $\overline\eta_f\in H^1_c(\Gamma_1(pN),\F(a^{k-2}))$ is defined by
$$
D\mapsto \frac{1}{\varpi^m p^r}\varphi_f(D)\bigg|_{(X,Y)=(0,1)}\Mod \varpi,
$$
where $\ord_p(\varpi^m p^r)=\mu_{\text{min}}^\pm(f)$ by construction. In particular, after scaling by a $p$-adic unit, we have 
$$
\varpi^{-\mu_{\text{min}}^\pm(f)}\Phi_{k_f}(\overline\varphi_f)=\overline\eta_f.
$$
Thus $\varpi^{-\mu_{\text{min}}^\pm(f)}\Phi_{k_f}(\overline\varphi_f)\in H^1_c(\Gamma_1(pN),\F(a^{k-2}))$ is nonzero, and since the same is true of $\Phi_{k}(\overline \varphi_g)$ by Corollary \ref{cor:nonzero-modp}, the multiplicity one result Corollary~\ref{cor:Ed} now implies that 
$$
\varpi^{-\mu_{\text{min}}^\pm(f)}\Phi_{k_f}(\overline\varphi_f)=C  \Phi_{k}(\overline \varphi_g). 
$$ 
for some nonzero constant $C$, which we can take to be 1 after scaling by a $p$-adic unit. 
\end{proof}

Similar to Corollary \ref{cor:ThmB}, we immediately obtain the following: 
\begin{corollary}\label{cor:ThmC}
    Let $f$ and $g$ be modular forms satisfying the hypotheses of Theorem \ref{thm:smallslope}. Then $\lambda(\Theta_{n,0}(f,\omega^i))=\lambda(\Theta_{n,0}(g,\omega^i))$ for all $n\ge0$.
\end{corollary}

Theorem~\ref{thmC}, which is a special case of the following theorem (after taking $i=0$), follows from combining Theorems~\ref{thm:FL} and \ref{thm:smallslope}.

\begin{theorem}\label{thm:small-slope}
     Let $f\in S_{k_f}(\Gamma,\overline{\Q_p})$ and $g\in S_{k_g}(\Gamma,\overline{\Q_p})$ be $p$-non-ordinary cuspidal eigen-newforms, and $i\in\{0,1,\dots,p-2\}$. Suppose that 
\begin{enumerate}
\item[(i)]  $2<k_g<p+1$ and $\mu(g,\sharp,\omega^i)=\mu(g,\flat,\omega^i)=0$,
\item[(ii)] $\overline\rho_f\cong \overline\rho_g$,
\item[(iii)]$\ord_p(a_p(f))<k_g-2.$
 \end{enumerate}
 Then for $n\gg0$, we have 
 $$
 \lambda\big(\Theta_n(f,\omega^i)\big)=(k_g-1)q_n+\lambda(g,\star,\omega^i),
 $$
 where $\star=\flat$ if $n$ is odd and $\star=\sharp$ if $n$ is even.
\end{theorem}

The reader is again referred to Tables~\ref{table3} and \ref{table4} for examples that illustrate this theorem. In the first block of Table~\ref{table3}, \texttt{G0N9k16A} does not satisfy condition (iii), and the $\lambda$-invariants of the corresponding Mazur--Tate elements exhibit different behavior from the forms of smaller slope that satisfy condition (iii). However, while the form \texttt{G0N9k28A} also does not satisfy condition (iii), its Mazur--Tate elements have the same $\lambda$-invariants as those of smaller slope. Similar phenomena can be observed in the other blocks of the two tables.

\section{Examples}\label{section:data}

The following tables contain $\lambda$-invariants of the Mazur--Tate elements attached to various modular forms $f$ illustrating the behaviors in Theorems \ref{thmA}, \ref{thmB}, and \ref{thmC}. 
The first and second columns of each table contain the LMFDB and Magma labels of the modular form, respectively. To initialize the form in Magma, use the command  \texttt{Newform("Label")}. The integer $d$ denotes the degree of the number field $K_f/\Q$ generated by the Fourier coefficients of the modular form. The integer $i$ denotes the index of the prime $\Pp_i|p$ of $K_f$ (ordered according to Magma's internal \texttt{Factorization(pOK)} function) with respect to which the slope and $\lambda$-invariants were computed. 

All data was computed using Magma \cite{Magma} and our code is available upon request. 

\subsection{Tables 1 and 2}
Tables \ref{table1} and \ref{table2} contain the $\lambda$-invariants of $\Theta_n(f)$ at $i=j=0$ for all cuspidal newforms $f$ of weight $k\leq 7$, degree $d\leq 6$, and level $\Gamma_1(N)$ with $N\leq 40$, that are non-ordinary at $p=5$ (Table 1) and $p=7$ (Table 2). 

For those $f$ satisfying the Fontaine--Laffaille condition $k\leq p$, the invariants $\lambda(\Theta_n(f))$ follow the pattern of Theorem \ref{thmA}, allowing us to  record (conjectural) values for the sharp/flat Iwasawa invariants of the corresponding $p$-adic $L$-functions. When $k> p$ and $a_p=0$, the Mazur--Tate elements follow the pattern of \cite{GL}, which similarly allows us to compute conjectural values for the signed $\lambda$-invariants. When $k=p+1$, the Serre weight is 2 and the behavior of the Mazur--Tate elements can be explained by \cite{PW} in terms of the invariants attached to a weight 2 modular form.

 It is interesting to note the several behaviors that occur in Table 1 at weight $k=p+2=7$. Already when $a_p=0$, it is observed in \cite{GL} that two separate behaviors can occur at weights $k\equiv p+2\Mod p^2-1$ depending on the size of $\lambda^{\sharp/\flat}$. For general positive slope at weight $k=p+2$, our data shows that a third behavior can also happen, which appears to occur when the slope is $<1$. 

 \begin{remark}\nf In both Tables 1 and 2, the signed $\lambda$-invariants at weights $2<k\leq p$ are as small as possible (see \eqref{eqn:Lflatbound}). We note that it can happen that $\lambda^\sharp>0$ and $\lambda^\flat>k-2$. For example,  the 5-non-ordinary rational cuspform $f\in S_4(\Gamma_0(528))$ with 
 LMFDB label \texttt{528.4.a.j} has Mazur-Tate elements with $\lambda(\Theta_n(f))=0,3,13,63,313,\dots$ for $n\in \{0,1,2,3,4\}$, suggesting that $\lambda^\sharp(f)=1$ and $\lambda^\flat(f)=3$.
 \end{remark}

\subsection{Tables 3 and 4} These tables are organized into blocks of modular forms which are congruent (in the sense of the following sentence) modulo $p=5$ (Table 3) and $p=7$ (Table 4). The first line of each block contains a $p$-non-ordinary rational newform $g$, and each of the remaining lines contains a higher weight newform $f$ satisfying $a_\ell(f)\equiv a_\ell(g)\Mod \Pp_i$
for all good primes $\ell\leq 100$.

\subsubsection{An interesting example}
Let $p=5$ and let $g$ and $f$ be the newforms \href{https://www.lmfdb.org/ModularForm/GL2/Q/holomorphic/27/4/a/b/}{\texttt{27.4.a.b}} and  \href{https://www.lmfdb.org/ModularForm/GL2/Q/holomorphic/27/16/a/b/}{\texttt{27.16.a.b}} from rows 1 and 4 of the second block of Table \ref{table3}. The form $f$ is congruent to $g$ at two distinct unramified primes $\Pp_1$ and $\Pp_2$ over 5. 
Computing the $\lambda$-invariants of $f$ with respect to the $\Pp_1$ and $\Pp_2$-adic valuations, we see from the table that 
$$
\lambda(\Theta_n(f))=\lambda(\Theta_n(g)) \quad \text{(with respect to $\Pp_1$)} 
$$
while 
$$
\lambda(\Theta_n(f))=\lambda(\Theta_{n-1}(g))+p^n-p^{n-1} \quad \text{(with respect to $\Pp_2$)}. 
$$
Since the slope of $f$ at $\Pp_1$ is $1/2<k_g-2$, the first equation follows from Corollary \ref{cor:ThmC}. Indeed, computer calculations suggest that $\mu(\Theta_n(f))=2$ for $n\geq 0$ and furthermore that
$$
p^{-2}\Theta_n(f)\equiv \Theta_n(g)\Mod \Pp_1,
$$
which can be seen as a consequence of Theorem \ref{thm:smallslope}.

On the other hand, $f$ has slope 3 at $\Pp_2$ and therefore neither Theorem \ref{thm:mediumweight} nor Theorem \ref{thm:smallslope} applies, since both the weight and slope are too large. In this case, computations suggest that instead we have the following congruence
$$
p^{-2}\Theta_n(f)\equiv  \nu_{n-1}^n\Theta_{n-1}(g)\Mod \Pp_2,
$$
where we recall that $\nu_{n-1}^n:\Oo[G_{n-1}]\rightarrow \Oo[G_n]$ is the corestriction map of Definition \ref{def:MT2}. There is a Hecke equivariant map
\[
\tilde\Phi_{16}:H^1_c(\Gamma_1(N),\overline V_{14})\rightarrow H^1_c(\Gamma_1(Np^3),\cO/p^2\Pp_i(a^2)),
\]
and the $\lambda$-invariant of $\Theta_n(f)$ can be calculated via the image $\tilde\Phi_{16}(p^{-2}\varphi_f^\pm)$. However, mod $p$ multiplicity one does not apply anymore, which explains why the $\lambda$-invariant might depend on the choice of $\Pp_i$.
This phenomenon seems to mirror the observation discussed in \cite[\S7]{PW}.

\newpage

\begin{table}[]
\caption{\label{table1} $\lambda$-invariants attached to newforms of level $N\leq 40$ and weight $k\leq 7$ that are non-ordinary at $p=5$.
}
    \centering
    \begin{minipage}[t]{0.1\textwidth}
        \scalebox{0.65}{
        \begin{tabular}{|c c c c c c c l l l l l c c|}
        \hline
\multicolumn{14}{|c|}{$p=5$} \\
            \hline
            LMFDB & Magma & $k$ & $N$ & $d$ &$i$ & slope & 0 & 1 & 2 & 3 & 4&  $\lambda^\sharp$ & $\lambda^\flat$ \\
            \hline
 \texttt{14.2.a.a}&\texttt{G0N14k2A}&2&14&1&1&$\infty$&0&0&4&20&104&0&0\\
 \texttt{18.2.c.a}&\texttt{G1N18k2A}&2&18&2&1&$\infty$&0&0&4&20&104&0&0\\
 \texttt{21.2.g.a}&\texttt{G1N21k2B}&2&21&2&1&$\infty$&0&0&4&20&104&0&0\\
\texttt{24.2.f.a} &\texttt{G1N24k2C}&2&24&2&1&$\infty$&0&0&4&20&104&0&0\\
\texttt{27.2.a.a} &\texttt{G0N27k2A}&2&27&1&1&$\infty$&0&0&4&20&104&0&0\\
 \texttt{28.2.d.a}&\texttt{G1N28k2C}&2&28&2&1&$\infty$&$\infty$&0&5&20&105&1&0\\
 \texttt{34.2.a.a}&\texttt{G0N34k2A}&2&34&1&1&$\infty$&0&0&4&20&104&0&0\\
 \texttt{36.2.a.a}&\texttt{G0N36k2A}&2&36&1&1&$\infty$&0&0&4&20&104&0&0\\
\texttt{37.2.a.b} &\texttt{G0N37k2B}&2&37&1&1&$\infty$&0&0&4&20&104&0&0\\
\texttt{38.2.a.a} &\texttt{G0N38k2A}&2&38&1&1&$\infty$&0&0&4&20&104&0&0\\
 \texttt{38.2.c.a}&\texttt{G1N38k2D}&2&38&2&1&$\infty$&0&0&4&20&104&0&0\\
\texttt{39.2.k.a} &\texttt{G1N39k2C}&2&39&4&1&$\infty$&0&0&4&20&104&0&0\\
 \texttt{39.2.k.a}&\texttt{G1N39k2C}&2&39&4&2&$\infty$&0&0&4&20&104&0&0\\
 \texttt{39.2.b.a}&\texttt{G1N39k2I}&2&39&2&1&$\infty$&$\infty$&0&5&20&105&1&0\\
\hline
\texttt{7.3.b.a} &\texttt{G1N7k3A}&3&7&1&1&$\infty$&0&1&8&41&208&0&1\\
 \texttt{8.3.d.a}&\texttt{G1N8k3A}&3&8&1&1&$\infty$&0&1&8&41&208&0&1\\
 \texttt{12.3.c.a}&\texttt{G1N12k3B}&3&12&1&1&$\infty$&0&1&8&41&208&0&1\\
\texttt{21.3.h.a} &\texttt{G1N21k3E}&3&21&2&1&$\infty$&0&1&8&41&208&0&1\\
\texttt{21.3.h.b} &\texttt{G1N21k3F}&3&21&4&1&1/2&0&1&8&41&208&0&1\\
\texttt{23.3.b.a} &\texttt{G1N23k3B}&3&23&3&1&$\infty$&0&1&8&41&208&0&1\\
 \texttt{27.3.b.a}&\texttt{G1N27k3C}&3&27&1&1&$\infty$&0&1&8&41&208&0&1\\
 \texttt{32.3.d.a}&\texttt{G1N32k3C}&3&32&1&1&$\infty$&0&1&8&41&208&0&1\\
\texttt{39.3.h.a} &\texttt{G1N39k3D}&3&39&2&1&$\infty$&0&1&8&41&208&0&1\\
 \texttt{39.3.i.a}&\texttt{G1N39k3G}&3&39&2&1&$\infty$&0&1&8&41&208&0&1\\
\hline
 \texttt{9.4.a.a}&\texttt{G0N9k4A}&4&9&1&1&$\infty$&0&2&12&62&312&0&2\\
 \texttt{12.4.b.a}&\texttt{G1N12k4B}&4&12&4&1&1/2&0&2&12&62&312&0&2\\
 \texttt{13.4.c.b}&\texttt{G1N13k4F}&4&13&4&1&1&0&2&12&62&312&0&2\\
\texttt{13.4.c.b} &\texttt{G1N13k4F}&4&13&4&2&1&0&2&12&62&312&0&2\\
\texttt{21.4.c.a} &\texttt{G1N21k4G}&4&21&2&1&$\infty$&0&2&12&62&312&0&2\\
 \texttt{24.4.f.a}&\texttt{G1N24k4C}&4&24&2&1&$\infty$&0&2&12&62&312&0&2\\
 \texttt{27.4.a.b}&\texttt{G0N27k4A}&4&27&1&1&1&0&2&12&62&312&0&2\\
 \texttt{27.4.a.a}&\texttt{G0N27k4B}&4&27&1&1&1&0&2&12&62&312&0&2\\
 \texttt{28.4.d.a}&\texttt{G1N28k4E}&4&28&2&1&$\infty$&0&2&12&62&312&0&2\\
 \texttt{32.4.a.c}&\texttt{G0N32k4B}&4&32&1&1&1&0&2&12&62&312&0&2\\
 \texttt{32.4.a.a}&\texttt{G0N32k4C}&4&32&1&1&1&0&2&12&62&312&0&2\\
 \texttt{39.4.f.a}&\texttt{G1N39k4H}&4&39&4&1&$\infty$&0&2&12&62&312&0&2\\
\texttt{39.4.f.a} &\texttt{G1N39k4H}&4&39&4&2&$\infty$&0&2&12&62&312&0&2\\
\hline
 \texttt{7.5.b.a}&\texttt{G1N7k5B}&5&7&1&1&$\infty$&0&3&16&83&416&0&3\\
\texttt{8.5.d.a} &\texttt{G1N8k5A}&5&8&1&1&$\infty$&0&3&16&83&416&0&3\\
 \texttt{8.5.d.b}&\texttt{G1N8k5B}&5&8&2&1&1/2&0&3&16&83&416&0&3\\
 \texttt{12.5.c.a}&\texttt{G1N12k5B}&5&12&1&1&$\infty$&0&3&16&83&416&0&3\\
\texttt{21.5.h.a} &\texttt{G1N21k5E}&5&21&2&1&$\infty$&0&3&16&83&416&0&3\\
 \texttt{23.5.b.a}&\texttt{G1N23k5B}&5&23&3&1&$\infty$&0&3&16&83&416&0&3\\
 \texttt{27.5.b.a}&\texttt{G1N27k5C}&5&27&1&1&$\infty$&0&3&16&83&416&0&3\\
\texttt{32.5.d.a} &\texttt{G1N32k5D}&5&32&1&1&$\infty$&0&3&16&83&416&0&3\\
 \texttt{32.5.d.b}&\texttt{G1N32k5E}&5&32&2&1&1/2&0&3&16&83&416&0&3\\
\texttt{39.5.h.a} &\texttt{G1N39k5D}&5&39&2&1&$\infty$&0&3&16&83&416&0&3\\
\texttt{39.5.i.a} &\texttt{G1N39k5G}&5&39&2&1&$\infty$&0&3&16&83&416&0&3\\
\hline
 \texttt{14.6.a.b}&\texttt{G0N14k6A}&6&14&1&1&1&0&4&20&104&520&&\\
\texttt{18.6.c.a} &\texttt{G1N18k6D}&6&18&4&2&1&0&4&20&104&520&&\\
 \texttt{21.6.g.a}&\texttt{G1N21k6E}&6&21&2&1&$\infty$&0&4&20&104&520&0&4\\
\texttt{24.6.f.a} &\texttt{G1N24k6E}&6&24&2&1&$\infty$&0&4&20&104&520&0&4\\
\texttt{27.6.a.a} &\texttt{G0N27k6A}&6&27&1&1&$\infty$&0&4&20&104&520&0&4\\
\texttt{28.6.d.a} &\texttt{G1N28k6F}&6&28&2&1&$\infty$&0&1&20&105&520&0&5\\
 \texttt{34.6.a.b}&\texttt{G0N34k6C}&6&34&2&1&1&0&4&20&104&520&&\\
 \texttt{36.6.a.b}&\texttt{G0N36k6A}&6&36&1&1&$\infty$&0&4&20&104&520&0&4\\
 \texttt{38.6.a.b}&\texttt{G0N38k6A}&6&38&1&1&1&0&4&20&104&520&&\\
 \texttt{39.6.k.a}&\texttt{G1N39k6E}&6&39&4&1&$\infty$&0&4&20&104&520&0&4\\
\texttt{39.6.k.a} &\texttt{G1N39k6E}&6&39&4&2&$\infty$&0&4&20&104&520&0&4\\
 \texttt{39.6.b.a}&\texttt{G1N39k6L}&6&39&4&3&1&0&0&20&105&520&&\\
\hline
        \end{tabular}}
    \end{minipage}
    \hfill
    \begin{minipage}[t]{0.4\textwidth}
    \vspace{-8.6cm}
        \scalebox{0.65}{
        \begin{tabular}{|c c c c c c c l l l l l c c|}
        \hline
\multicolumn{14}{|c|}{$p=5$} \\
            \hline
            LMFDB & Magma & $k$ & $N$ & $d$ &$i$ & slope & 0 & 1 & 2 & 3 & 4& $\lambda^\sharp$ & $\lambda^\flat$ \\
            \hline
\texttt{3.7.b.a}& \texttt{G1N3k7A}&7&3&1&1&$\infty$&0&4&24&124&624&0&4\\
\texttt{4.7.b.a}&\texttt{G1N4k7A}&7&4&2&1&1&0&4&24&124&624&&\\
\texttt{6.7.b.a}&\texttt{G1N6k7A}&7&6&2&1&1&0&1&5&25&125& & \\
\texttt{7.7.d.a}&\texttt{G1N7k7A}&7&7&2&1&1&0&4&24&124&624&&\\
\texttt{7.7.d.b}&\texttt{G1N7k7B}&7&7&4&1&2&0&1&5&25&125& & \\
\texttt{7.7.d.b}&\texttt{G1N7k7B}&7&7&4&2&2&0&1&5&25&125& & \\
\texttt{7.7.b.a}&\texttt{G1N7k7C}&7&7&1&1&$\infty$&0&4&24&124&624&0&4\\
\texttt{7.7.b.b}&\texttt{G1N7k7D}&7&7&2&1&1/2&0&1&8&41&208& & \\
\texttt{8.7.d.a}&\texttt{G1N8k7A}&7&8&1&1&$\infty$&0&4&24&124&624&0&4\\
\texttt{8.7.d.b}&\texttt{G1N8k7B}&7&8&4&1&3/4&0&1&8&41&208& & \\
\texttt{9.7.b.a}&\texttt{G1N9k7B}&7&9&2&1&1&0&1&5&25&125& & \\
\texttt{11.7.b.b}&\texttt{G1N11k7C}&7&11&4&1&1&0&4&24&124&624&&\\
\texttt{12.7.d.a}&\texttt{G1N12k7A}&7&12&6&1&1&0&1&5&25&125& & \\
\texttt{12.7.d.a}&\texttt{G1N12k7A}&7&12&6&3&1&0&1&5&25&125& & \\
\texttt{12.7.c.a}&\texttt{G1N12k7B}&7&12&2&1&1/2&0&1&8&41&208& & \\
\texttt{14.7.b.a}&\texttt{G1N14k7B}&7&14&4&1&1&0&1&5&25&125& & \\
\texttt{16.7.c.b}&\texttt{G1N16k7B}&7&16&2&1&2&0&1&5&25&125& & \\
\texttt{22.7.b.a}&\texttt{G1N22k7B}&7&22&6&1&1&0&1&5&25&125& & \\
\texttt{22.7.b.a}&\texttt{G1N22k7B}&7&22&6&2&1&0&1&5&25&125& & \\
\texttt{23.7.b.a}&\texttt{G1N23k7B}&7&23&1&1&$\infty$&0&1&5&25&125&1 &5 \\
\texttt{23.7.b.b}&\texttt{G1N23k7C}&7&23&2&1&$\infty$&0&4&24&124&624&0&4\\
\texttt{23.7.b.b}&\texttt{G1N23k7C}&7&23&2&2&$\infty$&0&1&5&25&125&1 &5 \\
\texttt{24.7.e.a}&\texttt{G1N24k7B}&7&24&6&2&1&0&1&5&25&125& & \\
\texttt{24.7.e.a}&\texttt{G1N24k7B}&7&24&6&3&1&0&1&5&25&125& & \\
\texttt{26.7.d.a}&\texttt{G1N26k7C}&7&26&6&1&1&0&1&5&25&125& & \\
\texttt{26.7.d.a}&\texttt{G1N26k7C}&7&26&6&2&1&0&4&24&124&624&&\\
\texttt{26.7.d.a}&\texttt{G1N26k7C}&7&26&6&3&1&0&1&5&25&125& & \\
\texttt{26.7.d.a}&\texttt{G1N26k7C}&7&26&6&4&1&0&1&5&25&125& & \\
\texttt{27.7.b.b}&\texttt{G1N27k7C}&7&27&2&1&1&0&1&5&25&125& & \\
\texttt{27.7.b.b}&\texttt{G1N27k7C}&7&27&2&2&1&0&1&5&25&125& & \\
\texttt{27.7.b.a}&\texttt{G1N27k7D}&7&27&2&1&1/2&0&1&8&41&208& & \\
\texttt{27.7.b.c}&\texttt{G1N27k7E}&7&27&4&1&2&0&1&5&25&125& & \\
\texttt{27.7.b.c}&\texttt{G1N27k7E}&7&27&4&4&2&0&2&6&26&126& & \\
\texttt{28.7.b.a}&\texttt{G1N28k7D}&7&28&4&3&1&0&1&5&25&125& & \\
\texttt{32.7.c.a}&\texttt{G1N32k7A}&7&32&2&1&2&0&1&5&25&125& & \\
\texttt{32.7.c.a}&\texttt{G1N32k7A}&7&32&2&2&2&0&1&5&25&125& & \\
\texttt{32.7.c.b}&\texttt{G1N32k7B}&7&32&4&1&1&0&2&6&26&126& & \\
\texttt{32.7.c.b}&\texttt{G1N32k7B}&7&32&4&4&1&0&2&6&26&126& & \\
\texttt{32.7.d.a}&\texttt{G1N32k7D}&7&32&1&1&$\infty$&0&2&6&26&126&2 &6 \\
\texttt{32.7.d.b}&\texttt{G1N32k7E}&7&32&4&1&3/4&0&1&8&41&208& & \\
\texttt{36.7.d.c}&\texttt{G1N36k7C}&7&36&2&1&1&0&1&5&25&125& & \\
\texttt{36.7.d.d}&\texttt{G1N36k7D}&7&36&4&1&1&0&1&5&25&125& & \\
\texttt{36.7.d.d}&\texttt{G1N36k7D}&7&36&4&2&1&0&1&5&25&125& & \\
\texttt{36.7.d.e}&\texttt{G1N36k7E}&7&36&6&1&1&0&1&5&25&125& & \\
\texttt{36.7.d.e}&\texttt{G1N36k7E}&7&36&6&3&1&0&1&5&25&125& & \\
\texttt{36.7.c.a}&\texttt{G1N36k7H}&7&36&2&1&1&0&1&5&25&125& & \\
\texttt{39.7.d.b}&\texttt{G1N39k7G}&7&39&2&1&$\infty$&0&1&5&25&125&1 &5 \\
\hline
        \end{tabular}}
    \end{minipage}
    \label{tab:iw_table}
\end{table}

\begin{table}[]
\caption{\label{table2} $\lambda$-invariants attached to newforms of level $N\leq 40$ and weight $k\leq 7$ that are non-ordinary at $p=7$.
}
    \centering
    \begin{minipage}[t]{-0.5\textwidth}
    	\vspace{-6.9cm}
        \scalebox{0.65}{
        \begin{tabular}{|c c c c c c c c c c c c c c|}
        \hline
\multicolumn{14}{|c|}{$p=7$} \\
            \hline
            LMFDB&Magma & $k$ & $N$ & $d$ &$i$ & slope & 0 & 1 & 2 & 3 & 4&  $\lambda^\sharp$ & $\lambda^\flat$ \\
            \hline
\texttt{13.2.e.a}&\texttt{G1N13k2A}&2&13&2&1&$\infty$&0&0&6&42&300&0&0\\
\texttt{13.2.e.a}&\texttt{G1N13k2A}&2&13&2&2&$\infty$&0&0&6&42&300&0&0\\
\texttt{15.2.a.a}&\texttt{G0N15k2A}&2&15&1&1&$\infty$&0&0&6&42&300&0&0\\
\texttt{20.2.e.a}&\texttt{G1N20k2B}&2&20&2&1&$\infty$&0&0&6&42&300&0&0\\
\texttt{24.2.a.a}&\texttt{G0N24k2A}&2&24&1&1&$\infty$&0&0&6&42&300&0&0\\
\texttt{24.2.f.a}&\texttt{G1N24k2C}&2&24&2&1&$\infty$&$\infty$&0&8&42&302&2&0\\
\texttt{32.2.a.a}&\texttt{G0N32k2A}&2&32&1&1&$\infty$&0&0&6&42&300&0&0\\
\texttt{33.2.d.a}&\texttt{G1N33k2E}&2&33&2&1&$\infty$&$\infty$&0&7&42&301&1&0\\
\texttt{34.2.c.b}&\texttt{G1N34k2C}&2&34&2&1&$\infty$&0&0&6&42&300&0&0\\
\texttt{34.2.b.a}&\texttt{G1N34k2E}&2&34&2&1&$\infty$&$\infty$&0&7&42&301&1&0\\
\texttt{36.2.b.a}&\texttt{G1N36k2D}&2&36&2&1&$\infty$&$\infty$&0&7&42&301&1&0\\
\hline
\texttt{8.3.d.a}&\texttt{G1N8k3A}&3&8&1&1&$\infty$&0&1&12&85&600&0&1\\
\texttt{11.3.b.a}&\texttt{G1N11k3B}&3&11&1&1&$\infty$&0&1&12&85&600&0&1\\
\texttt{15.3.d.b}&\texttt{G1N15k3C}&3&15&1&1&$\infty$&0&1&12&85&600&0&1\\
\texttt{15.3.d.a}&\texttt{G1N15k3D}&3&15&1&1&$\infty$&0&1&12&85&600&0&1\\
\texttt{16.3.c.a}&\texttt{G1N16k3A}&3&16&1&1&$\infty$&0&1&12&85&600&0&1\\
\texttt{20.3.f.a}&\texttt{G1N20k3B}&3&20&2&1&1&0&1&12&85&600&0&1\\
\texttt{20.3.d.c}&\texttt{G1N20k3E}&3&20&2&1&$\infty$&0&1&12&85&600&0&1\\
\texttt{23.3.b.a}&\texttt{G1N23k3B}&3&23&3&1&$\infty$&0&1&12&85&600&0&1\\
\texttt{23.3.b.a}&\texttt{G1N23k3B}&3&23&3&2&$\infty$&0&1&12&85&600&0&1\\
\texttt{32.3.d.a}&\texttt{G1N32k3C}&3&32&1&1&$\infty$&0&1&12&85&600&0&1\\
\texttt{36.3.d.b}&\texttt{G1N36k3A}&3&36&1&1&$\infty$&0&1&12&85&600&0&1\\
\texttt{36.3.d.a}&\texttt{G1N36k3B}&3&36&1&1&$\infty$&0&1&12&85&600&0&1\\
\texttt{39.3.d.b}&\texttt{G1N39k3I}&3&39&2&1&$\infty$&0&1&12&85&600&0&1\\
\texttt{39.3.d.a}&\texttt{G1N39k3J}&3&39&2&1&$\infty$&0&1&12&85&600&0&1\\
\texttt{39.3.d.c}&\texttt{G1N39k3K}&3&39&4&1&1/2&0&1&12&85&600&0&1\\
\hline
\texttt{17.4.a.a}&\texttt{G0N17k4A}&4&17&1&1&1&0&2&18&128&900&0&2\\
\texttt{20.4.e.a}&\texttt{G1N20k4B}&4&20&2&1&$\infty$&0&2&18&128&900&0&2\\
\texttt{22.4.a.a}&\texttt{G0N22k4C}&4&22&1&1&1&0&2&18&128&900&0&2\\
\texttt{24.4.f.a}&\texttt{G1N24k4C}&4&24&2&1&$\infty$&0&2&18&128&900&0&2\\
\texttt{26.4.a.b}&\texttt{G0N26k4A}&4&26&1&1&1&0&2&18&128&900&0&2\\
\texttt{26.4.b.a}&\texttt{G1N26k4G}&4&26&4&1&1/2&0&2&18&128&900&0&2\\
\texttt{31.4.a.b}&\texttt{G0N31k4B}&4&31&5&1&2&0&2&18&128&900&0&2\\
\texttt{32.4.a.b}&\texttt{G0N32k4A}&4&32&1&1&$\infty$&0&2&18&128&900&0&2\\
\texttt{33.4.d.a}&\texttt{G1N33k4I}&4&33&2&1&$\infty$&0&2&18&128&900&0&2\\
\texttt{34.4.c.b}&\texttt{G1N34k4G}&4&34&6&2&1&0&2&18&128&900&0&2\\
\texttt{36.4.b.a}&\texttt{G1N36k4E}&4&36&2&1&$\infty$&0&2&18&128&900&0&2\\
\texttt{37.4.a.b}&\texttt{G0N37k4B}&4&37&5&2&1&0&2&18&128&900&0&2\\
\texttt{39.4.a.b}&\texttt{G0N39k4B}&4&39&2&1&1/2&0&2&18&128&900&0&2\\
\texttt{39.4.j.b}&\texttt{G1N39k4F}&4&39&4&3&1&0&2&18&128&900&0&2\\
\texttt{39.4.j.b}&\texttt{G1N39k4F}&4&39&4&4&1&0&2&18&128&900&0&2\\
\hline
        \end{tabular}}
    \end{minipage}
    \hfill
    \begin{minipage}[t]{0.4\textwidth}
        \scalebox{0.65}{
        \begin{tabular}{|c c c c c c c c c c c c c c|}
        \hline
\multicolumn{14}{|c|}{$p=7$} \\
            \hline
            LMFDB &Magma & $k$ & $N$ & $d$ &$i$ & slope & 0 & 1 & 2 & 3 & 4& $\lambda^\sharp$ & $\lambda^\flat$ \\
            \hline
\texttt{4.5.b.a}& \texttt{G1N4k5A}&5&4&1&1&$\infty$&0&3&24&171&1200&0&3\\
\texttt{8.5.d.a}&\texttt{G1N8k5A}&5&8&1&1&$\infty$&0&3&24&171&1200&0&3\\
\texttt{9.5.b.a}&\texttt{G1N9k5B}&5&9&2&1&1&0&3&24&171&1200&0&3\\
\texttt{11.5.b.a}&\texttt{G1N11k5B}&5&11&1&1&$\infty$&0&3&24&171&1200&0&3\\
\texttt{15.5.d.b}&\texttt{G1N15k5C}&5&15&1&1&$\infty$&0&3&24&171&1200&0&3\\
\texttt{15.5.d.a}&\texttt{G1N15k5D}&5&15&1&1&$\infty$&0&3&24&171&1200&0&3\\
\texttt{23.5.b.a}&\texttt{G1N23k5B}&5&23&3&1&$\infty$&0&3&24&171&1200&0&3\\
\texttt{23.5.b.a}&\texttt{G1N23k5B}&5&23&3&2&$\infty$&0&3&24&171&1200&0&3\\
\texttt{25.5.c.c}&\texttt{G1N25k5C}&5&25&4&1&1/2&0&3&24&171&1200&0&3\\
\texttt{25.5.c.b}&\texttt{G1N25k5D}&5&25&4&1&1&0&3&24&171&1200&0&3\\
\texttt{25.5.c.b}&\texttt{G1N25k5D}&5&25&4&2&1&0&3&24&171&1200&0&3\\
\texttt{30.5.f.a}&\texttt{G1N30k5C}&5&30&4&1&1&0&3&24&171&1200&0&3\\
\texttt{30.5.f.a}&\texttt{G1N30k5C}&5&30&4&2&1&0&3&24&171&1200&0&3\\
\texttt{32.5.d.a}&\texttt{G1N32k5D}&5&32&1&1&$\infty$&0&3&24&171&1200&0&3\\
\texttt{36.5.d.a}&\texttt{G1N36k5A}&5&36&1&1&$\infty$&0&3&24&171&1200&0&3\\
\texttt{36.5.d.c}&\texttt{G1N36k5B}&5&36&4&1&1/2&0&3&24&171&1200&0&3\\
\texttt{39.5.d.b}&\texttt{G1N39k5I}&5&39&1&1&$\infty$&0&3&24&171&1200&0&3\\
\texttt{39.5.d.a}&\texttt{G1N39k5J}&5&39&1&1&$\infty$&0&3&24&171&1200&0&3\\
\texttt{39.5.d.c}&\texttt{G1N39k5K}&5&39&2&1&$\infty$&0&3&24&171&1200&0&3\\
\texttt{39.5.d.c}&\texttt{G1N39k5K}&5&39&2&2&$\infty$&0&3&24&171&1200&0&3\\
\texttt{40.5.l.a}&\texttt{G1N40k5B}&5&40&2&1&1&0&3&24&171&1200&0&3\\
\texttt{40.5.l.b}&\texttt{G1N40k5C}&5&40&4&2&1&0&3&24&171&1200&0&3\\
\hline
\texttt{17.6.a.b}&\texttt{G0N17k6A}&6&17&1&1&1&0&4&30&214&1500&0&4\\
\texttt{17.6.a.a}&\texttt{G0N17k6B}&6&17&1&1&2&0&4&30&214&1500&0&4\\
\texttt{20.6.e.a}&\texttt{G1N20k6B}&6&20&2&1&$\infty$&0&4&30&214&1500&0&4\\
\texttt{22.6.a.a}&\texttt{G0N22k6C}&6&22&1&1&2&0&4&30&214&1500&0&4\\
\texttt{24.6.f.a}&\texttt{G1N24k6E}&6&24&2&1&$\infty$&0&4&30&214&1500&0&4\\
\texttt{26.6.a.c}&\texttt{G0N26k6B}&6&26&2&1&1&0&4&30&214&1500&0&4\\
\texttt{26.6.b.a}&\texttt{G1N26k6H}&6&26&2&1&1&0&4&30&214&1500&0&4\\
\texttt{32.6.a.b}&\texttt{G0N32k6A}&6&32&1&1&$\infty$&0&4&30&214&1500&0&4\\
\texttt{33.6.d.a}&\texttt{G1N33k6I}&6&33&2&1&$\infty$&0&4&30&214&1500&0&4\\
\texttt{34.6.c.a}&\texttt{G1N34k6G}&6&34&6&1&1/3&0&4&30&214&1500&0&4\\
\texttt{36.6.b.a}&\texttt{G1N36k6E}&6&36&2&1&$\infty$&0&4&30&214&1500&0&4\\
\texttt{39.6.a.a}&\texttt{G0N39k6A}&6&39&1&1&1&0&4&30&214&1500&0&4\\
\hline
\texttt{8.7.d.a}&\texttt{G1N8k7A}&7&8&1&1&$\infty$&0&5&36&257&1800&0&5\\
\texttt{11.7.b.a}&\texttt{G1N11k7B}&7&11&1&1&$\infty$&0&5&36&257&1800&0&5\\
\texttt{15.7.d.b}&\texttt{G1N15k7C}&7&15&1&1&$\infty$&0&5&36&257&1800&0&5\\
\texttt{15.7.d.a}&\texttt{G1N15k7D}&7&15&1&1&$\infty$&0&5&36&257&1800&0&5\\
\texttt{16.7.c.a}&\texttt{G1N16k7A}&7&16&1&1&$\infty$&0&5&36&257&1800&0&5\\
\texttt{20.7.f.a}&\texttt{G1N20k7B}&7&20&6&2&1&0&5&36&257&1800&0&5\\
\texttt{20.7.d.c}&\texttt{G1N20k7E}&7&20&2&1&$\infty$&0&5&36&257&1800&0&5\\
\texttt{23.7.b.a}&\texttt{G1N23k7B}&7&23&1&1&$\infty$&0&5&36&257&1800&0&5\\
\texttt{23.7.b.b}&\texttt{G1N23k7C}&7&23&2&1&$\infty$&0&5&36&257&1800&0&5\\
\texttt{32.7.d.a}&\texttt{G1N32k7D}&7&32&1&1&$\infty$&0&5&36&257&1800&0&5\\
\texttt{36.7.d.b}&\texttt{G1N36k7A}&7&36&1&1&$\infty$&0&5&36&257&1800&0&5\\
\texttt{36.7.d.a}&\texttt{G1N36k7B}&7&36&1&1&$\infty$&0&5&36&257&1800&0&5\\
\texttt{39.7.d.c}&\texttt{G1N39k7H}&7&39&2&1&$\infty$&0&5&36&257&1800&0&5\\
\texttt{39.7.d.a}&\texttt{G1N39k7I}&7&39&2&1&$\infty$&0&5&36&257&1800&0&5\\
\hline
        \end{tabular}}
    \end{minipage}
    \label{tab:iw_table}
\end{table}

\begin{table}[]
\begin{center}
\caption{\label{table3} $\lambda$-invariants at $p=5$ attached to various newforms of the same level but different weights. Forms in the same block are congruent modulo 5. }
\scalebox{0.7}{
\begin{tabular}{|c|c||c|c|c|l|c||c|c|c|c|c|}
\hline
\multicolumn{7}{|c||}{$p=5$} & \multicolumn{5}{c|}{$\lambda(\Theta_n)$} \\
\hline
LMFDB &Magma & $k$ & $N$ & $d$ & $i$ & slope & 0 & 1 & 2 & 3 & 4 \\ 
\hline 
\hline
\href{https://www.lmfdb.org/ModularForm/GL2/Q/holomorphic/9/4/a/a/}{\texttt{9.4.a.a}}&\texttt{G0N9k4A}  & 4  & 9 & 1&1& $\infty$ & 0 & 2 & 12 & 62 & 312  \\
\href{https://www.lmfdb.org/ModularForm/GL2/Q/holomorphic/9/8/a/b/}{\texttt{9.8.a.b}}&\texttt{G0N9k8B}  & 8  & 9 & 2&1& 1/2 & 0 & 2 & 12 & 62 & 312   \\
\href{https://www.lmfdb.org/ModularForm/GL2/Q/holomorphic/9/12/a/c/}{\texttt{9.12.a.c}}&\texttt{G0N9k12C}  & 12  & 9 & 2&1& 1/2 &  0&  2&  12 & 62  & 312  \\
\href{https://www.lmfdb.org/ModularForm/GL2/Q/holomorphic/9/16/a/b/}{\texttt{9.16.a.b}}&\texttt{G0N9k16A} & 16 & 9 &1&1& $\infty$ &0 & 4 & 22 & 112 & 562  \\
\href{https://www.lmfdb.org/ModularForm/GL2/Q/holomorphic/9/16/a/e/}{\texttt{9.16.a.e}}&\texttt{G0N9k16E}  & 16  & 9 & 2&1&  1/2& 0 &2  &12 & 62 &  312  \\
\href{https://www.lmfdb.org/ModularForm/GL2/Q/holomorphic/9/20/a/d/}{\texttt{9.20.a.d}}&\texttt{G0N9k20D}  & 20  & 9 & 4&1&1/2  &  0& 2& 12 & 62 &  312  \\
\href{https://www.lmfdb.org/ModularForm/GL2/Q/holomorphic/9/20/a/d/}{\texttt{9.20.a.d}}&\texttt{G0N9k20D}  & 20  & 9 & 4&2& 2 & 0 & 2& 12 & 62 & 312  \\
\href{https://www.lmfdb.org/ModularForm/GL2/Q/holomorphic/9/20/a/d/}{\texttt{9.20.a.d}}&\texttt{G0N9k20D}  & 20  & 9 & 4&3& 2 &  0 & 2 & 12 & 62 & 312  \\
\href{https://www.lmfdb.org/ModularForm/GL2/Q/holomorphic/9/24/a/d/}{\texttt{9.24.a.d}}&\texttt{G0N9k24D}  & 24  & 9 & 4&1& 3/2 &  0 & 2 & 12 & 62 & 312  \\
\href{https://www.lmfdb.org/ModularForm/GL2/Q/holomorphic/9/24/a/d/}{\texttt{9.24.a.d}}&\texttt{G0N9k24D}  & 24  & 9 & 4&2&  1&  0& 2 & 12 & 62 &  312    \\
\href{https://www.lmfdb.org/ModularForm/GL2/Q/holomorphic/9/28/a/a/}{\texttt{9.28.a.a}}&\texttt{G0N9k28A} & 28 & 9 &1&1& $\infty$ &  0 & 2 & 12 & 62 & 312   \\
\href{https://www.lmfdb.org/ModularForm/GL2/Q/holomorphic/9/28/a/e/}{\texttt{9.28.a.e}}&\texttt{G0N9k28E}  & 28  & 9 & 4&1&  1/2& 0 &2  & 12 & 62 & 312   \\
\href{https://www.lmfdb.org/ModularForm/GL2/Q/holomorphic/9/28/a/e/}{\texttt{9.28.a.e}}&\texttt{G0N9k28E}  & 28  & 9 & 4&2& 3 & 0 & 2 & 12& 62 & 312   \\
\hline
\hline
\href{https://www.lmfdb.org/ModularForm/GL2/Q/holomorphic/27/4/a/b/}{\texttt{27.4.a.b}}&\texttt{G0N27k4A}  &  4 & 27 &1&1& 1& 0 & 2 & 12 & 62 & 312   \\
\href{https://www.lmfdb.org/ModularForm/GL2/Q/holomorphic/27/8/a/b/}{\texttt{27.8.a.b}}&\texttt{G0N27k8E}  & 8  & 27  & 2&1&1/2& 0 &  2&  12&  62&312   \\
\href{https://www.lmfdb.org/ModularForm/GL2/Q/holomorphic/27/12/a/e/}{\texttt{27.12.a.e}}&\texttt{G0N27k12D}  &12   & 27  &4&1& 1/2& 0  &2  & 12 & 62 &312   \\
\href{https://www.lmfdb.org/ModularForm/GL2/Q/holomorphic/27/16/a/b/}{\texttt{27.16.a.b}}&\texttt{G0N27k16C}  & 16  & 27 &5&1& 1/2& 0 & 2 &12  &62 & 312   \\
\href{https://www.lmfdb.org/ModularForm/GL2/Q/holomorphic/27/16/a/b/}{\texttt{27.16.a.b}}&\texttt{G0N27k16C}  & 16  & 27 &5&2& 3& 0 &  4& 22 &112  & 562  \\
\href{https://www.lmfdb.org/ModularForm/GL2/Q/holomorphic/27/20/a/b/}{\texttt{27.20.a.b}}&\texttt{G0N27k20E}  & 20  &27  & 6&1&2& 0  &  2&  12& 62  &312 \\
\href{https://www.lmfdb.org/ModularForm/GL2/Q/holomorphic/27/20/a/b/}{\texttt{27.20.a.b}}&\texttt{G0N27k20E}  & 20  &27  & 6&2 & 1/2 &  0& 2&  12&  62&312   \\
\href{https://www.lmfdb.org/ModularForm/GL2/Q/holomorphic/27/24/a/c/}{\texttt{27.24.a.c}}&\texttt{G0N27k24E}  &24   & 27 & 8&1&2&  0& 2 & 12 & 62&312  \\
\href{https://www.lmfdb.org/ModularForm/GL2/Q/holomorphic/27/24/a/c/}{\texttt{27.24.a.c}}&\texttt{G0N27k24E}  &24   & 27 & 8&4 & 1&0  &2  &12  &  62&312    \\
\href{https://www.lmfdb.org/ModularForm/GL2/Q/holomorphic/27/28/a/c/}{\texttt{27.28.a.c}}&\texttt{G0N27k28B}  & 28  & 27 & 9&1&1/2&  0&  2&12  &   62&312     \\
\href{https://www.lmfdb.org/ModularForm/GL2/Q/holomorphic/27/28/a/c/}{\texttt{27.28.a.c}}&\texttt{G0N27k28B}  & 28  & 27 & 9&4&3&  0& 2 &12  &  62&312     \\
\hline
\end{tabular}
}
\end{center}
\end{table}

\begin{table}[]
\begin{center}
\caption{\label{table4} $\lambda$-invariants at $p=7$ attached to various newforms of the same level but different weights. Forms in the same block are congruent modulo 7. }
\scalebox{0.7}{
\begin{tabular}{|c|c||c|c|c|c|c||c|c|c|c|}
\hline
\multicolumn{7}{|c||}{$p=7$} & \multicolumn{4}{c|}{$\lambda(\Theta_n)$} \\
\hline
LMFDB &Magma  & $k$ & $N$ & $d$ &$i$& slope & 0 & 1 & 2 & 3 \\ 
\hline
\hline
\href{https://www.lmfdb.org/ModularForm/GL2/Q/holomorphic/8/3/d/a/}{\texttt{8.3.d.a}}& \texttt{G1N8k3A}  &  3 &8  & 1&1&$\infty$ & 0 & 1 & 12 &85   \\
\href{https://www.lmfdb.org/ModularForm/GL2/Q/holomorphic/8/9/d/b/}{\texttt{8.9.d.b}}& \texttt{G1N8k9B}  &  9 &8  &6 &1&$1/2$ & 0 & 1 & 12 &85   \\
\href{https://www.lmfdb.org/ModularForm/GL2/Q/holomorphic/8/15/d/a/}{\texttt{8.15.d.a}}& \texttt{G1N8k15A} & 15 & 8 & 1 &1&$\infty$&  0& 6 & 43 & 306   \\
\href{https://www.lmfdb.org/ModularForm/GL2/Q/holomorphic/8/51/d/a/}{\texttt{8.51.d.a}}& \texttt{G1N8k51A} & 51 & 8 & 1 &1&$\infty$& 0& 1&12 & 85    \\
\hline
\hline
\href{https://www.lmfdb.org/ModularForm/GL2/Q/holomorphic/32/4/a/b/}{\texttt{32.4.a.b}}& \texttt{G0N32k4A}  & 4  & 32 & 1&1&$\infty$ & 0 & 2 & 18 & 128  \\
\href{https://www.lmfdb.org/ModularForm/GL2/Q/holomorphic/32/10/a/d/}{\texttt{32.10.a.d}}& \texttt{G0N32k10B}  & 10  & 32 &  2&1&1/2&0  & 2 & 18 & 128  \\
\href{https://www.lmfdb.org/ModularForm/GL2/Q/holomorphic/32/16/a/d/}{\texttt{32.16.a.d}}& \texttt{G0N32k16C}  & 16  & 32 &4 &1&1/2 & 0 & 2 &18  & 128 \\
\href{https://www.lmfdb.org/ModularForm/GL2/Q/holomorphic/32/22/a/a/}{\texttt{32.22.a.a}}&\texttt{G0N32k22A} & 22 & 32 & 1 &1&$\infty$& 0 & 6 & 44 & 312  \\
(not in LMFDB)&\texttt{G0N32k52A} & 52 & 32 & 1 &1&$\infty$& 0 & 2  & 18 & 128   \\
\hline
\hline
\href{https://www.lmfdb.org/ModularForm/GL2/Q/holomorphic/4/5/b/a/}{\texttt{4.5.b.a}}& \texttt{G1N4k5A}  & 5  & 4 & 1&1&$\infty$ & 0 & 3& 24 & 171  \\
\href{https://www.lmfdb.org/ModularForm/GL2/Q/holomorphic/4/11/b/a/}{\texttt{4.11.b.a}}& \texttt{G1N4k11A}  & 11  & 4 & 4&2&1/2 & 0 & 3&24  &171    \\
\href{https://www.lmfdb.org/ModularForm/GL2/Q/holomorphic/4/17/b/b/}{\texttt{4.17.b.b}}& \texttt{G1N4k17B}  & 17  & 4 & 6&1&1/2 & 0 &3 &24  &171    \\
\href{https://www.lmfdb.org/ModularForm/GL2/Q/holomorphic/4/23/b/a/}{\texttt{4.23.b.a}}& \texttt{G1N4k23A}  & 23  & 4 & 10&3&1/2 &0  &3 & 24 & 171  \\
\href{https://www.lmfdb.org/ModularForm/GL2/Q/holomorphic/4/29/b/a/}{\texttt{4.29.b.a}}& \texttt{G1N4k29A} & 29 & 4 & 1 &1&$\infty$& 0 & 6 & 45 & 318   \\
\href{https://www.lmfdb.org/ModularForm/GL2/Q/holomorphic/4/29/b/b/}{\texttt{4.29.b.b}}& \texttt{G1N4k29B} & 29 & 4 & 12 &1&1/2& 0 & 3 & 24 & 171   \\
\href{https://www.lmfdb.org/ModularForm/GL2/Q/holomorphic/4/53/b/a/}{\texttt{4.53.b.a}}& \texttt{G1N4k53A} & 53 & 4 &1  &1&$\infty$& 0 & 3 & 24 & 171   \\
\hline
\hline
\href{https://www.lmfdb.org/ModularForm/GL2/Q/holomorphic/32/6/a/b/}{\texttt{32.6.a.b}}& \texttt{G0N32k6A}  &  6 &32  & 1&1&$\infty$ & 0 & 4& 30 &214    \\
\href{https://www.lmfdb.org/ModularForm/GL2/Q/holomorphic/32/12/a/c/}{\texttt{32.12.a.c}}& \texttt{G0N32k12C} & 12&32  &2 &1&1/2 &  0 & 4& 30 & 214  \\
\href{https://www.lmfdb.org/ModularForm/GL2/Q/holomorphic/32/18/a/e/}{\texttt{32.18.a.e}}& \texttt{G0N32k18D} & 18 &32  &4  &2&1/2& 0 & 4&  30 &214    \\
\href{https://www.lmfdb.org/ModularForm/GL2/Q/holomorphic/32/24/a/d/}{\texttt{32.24.a.d}}& \texttt{G0N32k24C} & 24 &32  & 6&2&1/2 & 0 &4 & 30 & 214   \\
\href{https://www.lmfdb.org/ModularForm/GL2/Q/holomorphic/32/30/a/b/}{\texttt{32.30.a.b}}& \texttt{G0N32k30B}& 30 & 32 &6 &2&1/2 & 0 & 4  &30 &214    \\
(not in LMFDB)& \texttt{G0N32k36A}& 36 & 32 & 1&1&$\infty$ & 0 & 6  &46 & 324   \\ \hline
\end{tabular}
}
\end{center}
\end{table}

\clearpage

\bibliographystyle{alpha}
\bibliography{references}

\end{document}